\documentclass[a4paper, oneside, reqno]{amsart}
\usepackage{graphicx} % Required for inserting images
\usepackage{amsmath,amsfonts,amssymb}
\usepackage{amsthm, thmtools}
\usepackage{mathtools}
\usepackage{algorithm}
\usepackage{algpseudocode}
\usepackage{hyperref}
\usepackage{enumitem}
\usepackage{aligned-overset}
\usepackage[svgnames]{xcolor}
\usepackage{todonotes}
\usepackage{tikz}
\usetikzlibrary{calc}

\newcommand{\abs}[1]{\left\lvert#1\right\rvert}
\newcommand{\norm}[1]{\left\|#1\right\|}

\newcommand{\N}{\mathbb{N}}

\newcommand{\R}{\mathbb{R}}

\declaretheorem[name=Theorem]{theorem}
\declaretheorem[name=Lemma]{lemma}
\declaretheorem[name=Corollary]{corr}
\theoremstyle{remark}
\declaretheorem[name=Remark]{remark}

\declaretheoremstyle[
    headfont=\bfseries,
    headpunct={.},
    postheadspace={0.5em},
    notefont=\bfseries, % If you add a note, it stays bold
    notebraces={}{},     % Removes parentheses from notes
]{alphastyle}

\newtheorem{assumption}{Assumption}
\renewcommand{\theassumption}{\Alph{assumption}}
\title[A priori adaptive numerical methods for blow-up times of ODEs]{A priori adaptive numerical methods for estimating blow-up times of autonomous ODEs}

\author{H{\aa}kon Hoel}
\address{Department of Mathematics, University of Oslo, Norway}
\email{haakonah@math.uio.no}
\author{Johannes Vincent Meo$^\star$}
\address{Department of Mathematics, University of Oslo, Norway}
\email{johannvm@math.uio.no}

\thanks{$^\star$Corresponding author: Johannes Vincent Meo (johannvm@math.uio.no)}

\begin{document}
\begin{abstract}
We present effective a priori adaptive numerical methods for estimating the blow-up time for solutions of autonomous ODEs.
The novelty of our approach is to base our adaptive steps on the sensitivity of an auxiliary hitting time.
We provide results on the theoretical error rates, and show that there is a benefit in terms of computational effort in choosing our adaptive algorithm over alternative approaches. 
Numerical experiments support our theoretical results and show how the methods perform in practice. 
\end{abstract}

\maketitle

\section{Introduction}
For $n$-dimensional autonomous ODE \(x'(t) = b(x(t))\) where solutions blow up in finite time,
this work is concerned with developing effective numerical methods for approximation of the \emph{blow-up time}
\begin{equation*}
    \tau(x_0) \coloneqq \inf\Big\{t \geq 0 \mid \lim_{s\uparrow t}\abs{x(s)} = \infty, \; x(0) = x_0\Big\}.
\end{equation*}
Blow-up times are encountered in a wide variety of physical, engineering and biology models, and indicate phenomenons such as thermal runaway in combustion~\cite{Bebernes1989}, collapse in nonlinear optics~\cite{McLaughlin1986} and extreme concentration of cells in chemotaxis~\cite{Winkler2013}.  Moreover, in the one-dimensional setting, the blow-up time can be represented as the improper integral 
\[\
\tau = \int_{x_0}^\infty\frac{1}{b(x)}\,dx,
\]
which for instance can represent the tail probability of a random variable with probability density function $1/b$. 

Our approach has three main pillars that can be summarized as follows: 1. For a given tolerance $\epsilon>0$,
introduce the auxiliary hitting time 
\[
\tau_\epsilon = \inf\Big\{t \geq 0 \mid \lim_{s\uparrow t}\abs{x(s)} = r(\epsilon), \; x(0) = x_0\Big\},
\]
where $r(\epsilon)>0$ is a threshold parameter chosen sufficiently large so that $|\tau- \tau_\epsilon| = \mathcal{O}(\epsilon)$. 
2. Solve the ODE numerically and compute the hitting time $\bar \tau$ when the norm of your numerical solution attains the value $r(\epsilon)$. 
3. Apply adaptive time stepping with sufficiently small steps to ensure that $|\bar \tau - \tau_\epsilon| = \mathcal{O}(\epsilon)$, but still large enough to preserve efficiency. 

To simplify theoretical studies, we apply the first-order forward Euler scheme as numerical integrator (see Remark~\ref{rem:ho_sceme} for extensions to higher-order integrators). 
As an extension of classic works on adaptive methods on bounded domains~\cite{Bangerth2003,Chaudhry2021,Eriksson1995,Moon2003}, our adaptive time stepping is based on equilibrating error indicators that are the local error of the integrator multiplied with the sensitivity of the threshold hitting time $\tau_\epsilon$ w.r.t. perturbations in the initial condition.
To our knowledge, this is the first sensitivity based numerical method with the focus on \(\tau_\epsilon\) rather than the solution \(x\).
Under relatively mild conditions on \(b\), Theorems~\ref{thm:error_alg1} and~\ref{thm:propertiesAlgRn} present rigorous asymptotic bounds on the approximation error and computational cost, showing that \(|\tau - \bar \tau_\epsilon| = \mathcal{O}(\epsilon)\)  is achieved at the computational cost \(\mathcal{O}(\epsilon^{-1})\) (where the cost measure is proportional to the algorithm's FLOPS).
Rigorous asymptotic bounds of this kind are rare in the literature on numerical methods for estimating blow-up times, and we show that our method
yields better efficiency than first-order numerical methods using uniform time steps and adaptive methods that do not account for the sensitivity, cf.~Sections~\ref{sec:adaptiveVsUniform1D} and~\ref{sec:numExp}.

The literature contains several numerical strategies for estimating blow-up times, with a comprehensive survey in \cite{Bandle1998}. 
Stuart and Floater~\cite{Stuart90} propose a transformation of the underlying time-variable, dependent on the size of the current solution, in order to get the correct asymptotic behaviour when using a  first-order implicit Euler scheme. Hirota and Ozawa~\cite{Hirota2006} utilizes an arc-length transformation to extend the computationally very useful bijection between time and solution-state for one-dimensional problems to higher dimensions and applies a higher-order Runge Kutta methods combined with local-error based adaptivity to achieve an experimentally very efficient method, with some need for parameter tweaking. A posteriori adaptive finite-element methods for inhomogeneous ODEs on abstract Hilbert spaces are considered in~\cite{Holm2018,Cangiani2016}. Our adaptive approach complements these contributions by using a sensitivity-based a priori numerical method for which rigorous convergence rates hold for both the approximation error and the computational cost. 

For semi-linear parabolic PDEs, the seminal work by~\cite{Nakagawa1975} shows that an appropriately 
scaled finite difference method converges to the true blow-up time asymptotically. 
Berger and Kohn~\cite{Berger1988} introduced an algorithm that rescales the numerical solution when it reaches a predefined threshold value \(M\). Rescaling produces a new problem that is solved at finer resolution near the blow-up time and the blow-up point in space. Iterating this procedure leads to an adaptive and localized numerical method. 
Convergence analysis of the rescaling method was studied by Cho and Sun~\cite{Cho2025} and the rescaling method was extended to ODEs by Cho and Wu in ~\cite{Cho2024}, where a first order convergence rate is proven for one-dimensional ODEs with $b(x) = x^p$ for $p>1$.  

Lastly, the works~\cite{Matsue2020,Takayasu2017} study the related problem of identifying with numerical methods when true solutions of dynamical systems blow up. Although not relevant for this work, as we impose sufficient growth assumptions on the vector field to ensure blow ups, this is a highly relevant challenge for applications. 

The rest of this paper is organized as follows. Section~\ref{sec:adaptive1D} introduces the 
mathematical problem of estimating blow-up times in 1D, describes our adaptive method in Algorithm~\ref{alg:1d} and theoretical results on convergence and computational cost 
in Theorem~\ref{thm:error_alg1}. Section~\ref{sec:higherDim} extends our adaptive method to 
higher-dimensional ODEs. The resulting method is described in Algorithm~\ref{alg:Rn}
and Theorem~\ref{thm:propertiesAlgRn}. 
Finally, Section~\ref{sec:numExp} presents several numerical examples backing up our theoretical results. 

\section{An adaptive numerical method for estimation of blow-up times for 1D ODEs}
\label{sec:adaptive1D}
In this section we introduce an a priori adaptive method for one dimensional ODEs and provide asymptotic rates for the approximation error and computational cost in Theorem \ref{thm:error_alg1}. Subsequent remarks shortly discuss possible extensions to separable ODEs and use of higher-order integrators. Lastly, we compare our approach to using uniform steps in Section \ref{sec:adaptiveVsUniform1D}.

We consider the autonomous one-dimensional ODE
\begin{equation}
\label{eq:ode}
\begin{split}
    x'(t) &= b(x(t)), \qquad t \in (0, \tau)\\
    x(0) &= x_0,
\end{split}
\end{equation}
for right-hand sides that satisfy the following assumption. 
\begin{assumption}
The function \(b:D = [x_0, \infty) \to (0,\infty)\) with \(x_0>0\) satisfies
\begin{enumerate}[label=\arabic*), ref=\theassumption.\arabic*]
    \label{asmp:b}
    \item \(b \in C^1(D)\) with \(b'\) positive and increasing on \(D\). \label{asmp:b.C1_bDerPos}
    \item There exist positive monotone decreasing functions \(F, G\in C^1((0,\infty))\) such that \(-\frac{d}{dx}G(b(x)) \leq\frac{1}{b(x)} \leq -\frac{d}{dx}F(b(x))\) for all \(x\in D\) with 
    \mbox{\(\underset{x\to\infty}{\lim}F(x)=0\)} and \(-\frac{1}{F^{-1}(\epsilon)G'(F^{-1}(\epsilon))}=\mathcal{O}(\epsilon^{\alpha})\) for some \(\alpha>-2\) for sufficiently small \(\epsilon>0\). \label{asmp:b.FandG}
    \item There exist some \(k>1\) such that \(\int_{x_0}^\infty\frac{\sqrt{b'(kx)}}{b(x)} \, dx <\infty\). \label{asmp:b.integrabilityCond}
    \item The integral \(\tau = \int_{x_0}^\infty \frac{1}{b(x)}dx\) is finite. \label{asmp:b.tauFinite}
\end{enumerate}
\end{assumption}

\begin{remark}[Alternative to Assumption~\ref{asmp:b.FandG}]
    \label{rem:altAsmpThm1}
    An alternative to Assumption \ref{asmp:b.FandG} is to assume that \(b\in C^2(D)\) with \(b''(x) > 0\) and \(x \leq b'(x)^C\) for \(x \in D\) and some constant \(C > 0\). We will later on, in Remark \ref{rem:followAltAsmpThm1}, show that our main result for this section, Theorem \ref{thm:error_alg1}, also holds under this alternative assumption. 
\end{remark}

As $b$ is locally Lipschitz on $D$ we recall from~\cite[Chapter 2]{Walter1998} 
that~\eqref{eq:ode} has a unique solution \(x \in C^1([0,\tau))\), and 
Assumption~\ref{asmp:b.tauFinite} ensures a finite blow-up time. 

To account for how the dynamics depends on the initial condition, 
we let \(x(t; \tilde x_0, t_0)\) denote the solution of
\begin{equation*}
\begin{split}
    x'(t; \tilde x_0, t_0) &= b(x(t; \tilde x_0, t_0)),\qquad t\in(t_0,\tau)\\
    x(t_0) &= \tilde x_0,
\end{split}
\end{equation*}
where \(\tilde x_0\in [x_0, \infty)\) and \(t_0 \in [0, \infty)\). And for \(r\in [x_0, \infty)\), we 
set
\begin{equation}\label{eq:tau_1D}
\tau_r(\tilde x_0, t_0) = t_0 + \int_{\min\{\tilde x_0, r\}}^r\frac{1}{b(x)}\,dx,
\end{equation}
which then denotes the first time $t \ge t_0$ that $x(t; \tilde x_0, t_0)$ exceeds the value \(r\) (equivalently, the hitting time of \(r\) for the process \(x(t; \tilde x_0, t_0)\), when assuming that \(x_0 \leq r\)). 
We also introduce the short-hand notation \(\tau_r = \tau_r(x_0, 0)\), denoting the hitting time of \(r\) for \(x(t)\).

Let \(\bar x\) denote the forward Euler numerical solution of \(x(t)\), given by 
the scheme 
\begin{equation*}
    \bar{x}_{n+1} = \bar{x}_{n} + b(\bar{x}_{n})h_{n} \text{ for } n=0, 1, \hdots, N-1,
\end{equation*}
where \(h_{n} >0\) is an adaptive time step, \(\bar x_0 = x_0\), and \(N \in \N\) is the number of steps. 
In our case, \(N\) will be the smallest integer such that \(\bar x_N > r\) where \(r\) is chosen to control the error \(\abs{\tau-\tau_r}\) for a given tolerance \(\epsilon\). 
The iteration \(t_{n+1} = t_{n} + h_{n}\) for \(n = 0, 1, \hdots, N-1\) then defines the grid for our numerical solution, with \(\bar\tau \coloneqq t_N\) being our approximation of \(\tau_r\). 
As a consequence of Assumption~\ref{asmp:b}, we note that  \(\bar x_{n+1} \leq x(t_{n+1})\): If we assume that \(\bar x_n = x(t_n)\), then  
\begin{equation}
\label{eq:xBarBndX}
    \bar{x}_{n+1} = \bar x_n + h_nb(\bar x_n) = x(t_n) + h_nb(x(t_n)) \leq x(t_n) + \int_{t_n}^{t_{n+1}}b(x(t))\,dt = x(t_{n+1}),
\end{equation}
as \(b\) is monotone increasing. 

Our a priori adaptive method is presented in Algorithm \ref{alg:1d}. 
It estimates \(\tau\) by numerically approximating the auxiliary hitting-time \(\tau_r\). 
We chose \(r\) such that \(b(r) = F^{-1}(\epsilon)\) holds, as we from \eqref{eq:firstErrorTerm} have that this is sufficient to ensure \(\abs{\tau-\tau_r} \leq \epsilon\). 
The ODE is then solved using forward Euler with the adaptive step length \(h_n = \sqrt{\frac{\epsilon^2}{b'(\min\{k\bar x, r\})}}\) until it first passes the value \(r\). 
The time when \(r\) is reached, \(t_N\), is then our estimate \(\bar \tau\) of \(\tau_r\). 
The adaptive time steps \(h_n\) are chosen as they ensures \(\abs{\bar\tau - \tau_r} = \mathcal{O}(\epsilon)\), see proof of Theorem \ref{thm:error_alg1}.

\begin{algorithm}
    \caption{Estimating blow-up time of 1D ODE}
    \label{alg:1d}
    \begin{algorithmic}[1]
        \Require \(b,\, x_0,\, k,\, F^{-1},\,\epsilon > 0\).
        \State Set \(t = 0,\, N=0,\,\bar x = x_0\) and let \(r\) be such that \(b(r) = F^{-1}(\epsilon)\).
        \While{\(\bar x < r\)}
            \State \(h = \sqrt{\frac{\epsilon^2}{b'(\min\{k\bar x,\,r\})}}\).
            \State \(\bar x \leftarrow \bar x + b(\bar x)h\).
            \State \(t \leftarrow t + h\).
            \State \(N \leftarrow N+1\).
        \EndWhile
        \State \(\bar\tau = t\).
        \Ensure \(\bar\tau,\,N\).
    \end{algorithmic}
\end{algorithm}
As the number of floating point operations per iteration in Algorithm \ref{alg:1d} is bounded by a small fixed integer \(C\), we have that the total number of operations are \(\mathcal{O}(N)\). We will therefore define the computational cost to be the number of iterations \(N\).
The main result on the performance of our 1D numerical method is stated below. 

\begin{theorem}[Properties of Algorithm \ref{alg:1d}]
    \label{thm:error_alg1}
    Let Assumption~\ref{asmp:b} hold and let \(\bar\tau\) be calculated using Algorithm~\ref{alg:1d}.
    Then, for sufficiently small \(\epsilon > 0\), it holds that \(\abs{\bar\tau-\tau} = \mathcal{O}(\epsilon)\) and the computational cost of the method is \(\mathcal{O}\left(\epsilon^{-1}\right)\).
\end{theorem}
\begin{proof}
    For any \(r \geq x_0\) we have
    \begin{equation}
        \label{eq:errorSplit}
        \abs{\bar\tau-\tau} \leq \abs{\bar\tau-\tau_r} + \abs{\tau_r-\tau}.
    \end{equation}
    To prove our claim, it suffices to show both terms on the right-hand side are \(\mathcal{O}(\epsilon)\). 
    From Algorithm \ref{alg:1d} we have that \(r(\epsilon)\) is chosen such that \(b(r(\epsilon)) = F^{-1}(\epsilon)\), where we write \(r\) as a function of \(\epsilon\) to emphasize its dependence of the tolerance.
    For the second term, we then get
    \begin{equation}
        \label{eq:firstErrorTerm}
            \abs{\tau_{r(\epsilon)}-\tau} = \int_{r(\epsilon)}^\infty\frac{1}{b(x)}\,dx \stackrel{\ref{asmp:b.FandG}}{\leq} \int_{r(\epsilon)}^\infty-\frac{d}{dx}F(b(x)) \,dx = F(b(r(\epsilon)))=F(F^{-1}(\epsilon))=\epsilon.
    \end{equation}
    For later reference, from the lower bound of \(\frac{1}{b(x)}\) in Assumption \ref{asmp:b.FandG} we have that \(-\frac{d}{dx}G(b(x)) = -G'(b(x))b'(x) \leq \frac{1}{b(x)}\), thus by setting \(x = r(\epsilon)\) we get
    \begin{equation}
    \label{eq:boundB_der_r}
        b'(r(\epsilon)) \leq -\frac{1}{b(r(\epsilon))G'(b(r(\epsilon)))} = -\frac{1}{F^{-1}(\epsilon)G'(F^{-1}(\epsilon))} \stackrel{\ref{asmp:b.FandG}}{=} \mathcal{O}(\epsilon^\alpha),
    \end{equation}
    for some \(\alpha>-2\).

    For the first term of \eqref{eq:errorSplit}, we have
    \begingroup 
    \allowdisplaybreaks
    \begin{align}
    \label{eq:secondTermExp}
        \abs{\bar\tau-\tau_{r(\epsilon)}} &= \abs{\tau_{r(\epsilon)}(\bar{x}_N, t_N)-\tau_{r(\epsilon)}(\bar{x}_0, t_0)}\nonumber\\
        & \leq \sum_{n=1}^{N-1}\abs{\tau_{r(\epsilon)}(\bar{x}_n, t_n) - \tau_{r(\epsilon)}(\bar{x}_{n-1}, t_{n-1})} \nonumber \\
        &\hspace{6.5mm} + \abs{\tau_{r(\epsilon)}(\bar{x}_N, t_N) - \tau_{r(\epsilon)}(\bar{x}_{N-1}, t_{N-1})} \nonumber\\
        &\stackrel{\eqref{eq:tau_1D}}{\leq} \sum_{n=1}^{N-1}\abs{h_{n-1} - \int_{\bar x_{n-1}}^{\bar x_n}\frac{1}{b(x)}\, dx} + \abs{h_{N-1} - \int_{\bar x_{N-1}}^{r(\epsilon)}\frac{1}{b(x)}\, dx} \nonumber \\
        &\hspace{-1mm}\stackrel{\eqref{asmp:b.C1_bDerPos}}{\leq} \sum_{n=1}^{N-1}\abs{\left(\frac{1}{b(\bar x_{n-1})}-\frac{1}{b(\bar x_{n})}\right)b(\bar x_{n-1})h_{n-1}} + h_{N-1} \nonumber \\
        &\leq \sum_{n=1}^{N-1} b'(\bar x_n)h_{n-1}^2 + h_{N-1}.
    \end{align}
    \endgroup
    In the last inequality we take a first-order Taylor expansion of \(b(\bar x_n)\) around \(\bar x_{n-1}\) and use the monotonicity of $b'$ to $\max_{\theta \in [\bar x_{n-1}, \bar x_n]} |b'(\theta)| \le |b'(\bar x_n)|$.
    Moreover, one can show the the last line above is proportional to the sensitivity \(\partial_{x_0} \tau(r, x_0)\) times the local error in \(x\).
    
    The adaptive step length used in Algorithm \ref{alg:1d} gives \(h_{n-1} = \sqrt{\frac{\epsilon^2}{b'(\min\{k\bar x_{n-1},\,r(\epsilon)\})}}\), but for technical reasons, let us now introduce the following alternative step length:
    \begin{align}
        \label{eq:restrictions_h}
        g_{n-1} &= \min\Bigg\{ \underbrace{\sqrt{\frac{\epsilon^2}{b'(\min\{k\bar x_{n-1},\,r(\epsilon)\})}}}_{h_{n-1}}, \underbrace{\frac{(k-1)\bar x_{n-1}}{b(\bar x_{n-1})}}_{\tilde g_{n-1}} \Bigg\},
    \end{align}
    where \(k > 1\) is some value such that Assumption \ref{asmp:b.integrabilityCond} is satisfied. 
    Since \(g_{n-1} \leq \tilde g_{n-1}\), we have that
    \[\bar x_n = \bar x_{n-1} + g_{n-1}b(\bar x_{n-1}) \leq k\bar x_{n-1}.\] 
    Combining this with~\eqref{eq:secondTermExp} 
    and using that \(\bar x_{n} < r(\epsilon)\) for \(n < N\) yields
    \[\abs{\bar\tau-\tau_{r(\epsilon)}} \leq \sum_{n=1}^{N-1}b'(\min\{k\bar x_{n-1},\, r(\epsilon)\})g_{n-1}^2 + g_{N-1}\]
    Next, the other constraint on \(g_{n-1}\), \(h_{n-1}\), yields
    \begin{align}
        \label{eq:secondErrorTermFinal}
        \abs{\bar\tau-\tau_{r(\epsilon)}} \leq N\epsilon^2 + \mathcal{O}(\epsilon),
    \end{align}
    and it remains to bound \(N\). We first need to determine which of the constraints on \(g_{n-1}\) in \eqref{eq:restrictions_h} that is the most restrictive. For the quotient \(\frac{h_{n-1}}{\tilde g_{n-1}}\), we get
    \begin{align*}
        \frac{\sqrt{\frac{\epsilon^2}{b'(\min\{k\bar x_{n-1},\, r(\epsilon)\})}}}{\frac{(k-1)\bar x_{n-1}}{b(\bar x_{n-1})}} &= \frac{\epsilon b(\bar x_{n-1})}{(k-1)\bar x_{n-1}\sqrt{b'(\min\{k\bar x_{n-1},\, r(\epsilon)\})}} \\
        &\leq \frac{\epsilon c\bar x_{n-1} b'(\bar x_{n-1})}{(k-1)\bar x_{n-1}\sqrt{b'(\min\{k\bar x_{n-1},\, r(\epsilon)\})}} \\
        &\leq \frac{\epsilon c \sqrt{b'(r(\epsilon))}}{k-1} \\
        &\overset{\eqref{eq:boundB_der_r}}{=} \mathcal{O}(\epsilon^{1+\frac{\alpha}{2}}),
    \end{align*}
    where the transition to the second line uses that there exists a \(c>0\) such that \(b(x) = b(x_0) + b'(x)x \leq cxb'(x)\) for all \(x \in D\). 
    Thus for sufficiently small \(\epsilon\) we have that \(g_{n-1} = h_{n-1}\).

    Introducing the continuous extensions \(\bar x(t) = \bar x_{n-1}\) and \(h(t) = h_{n-1}\) for \(t \in [t_{n-1}, t_n)\), we obtain for sufficiently small \(\epsilon >0\) that
    \begingroup 
    \allowdisplaybreaks
    \begin{align}
    \label{eq:costEst1D}
        N &= \sum_{n=1}^{N}\frac{h_{n-1}}{h_{n-1}} \nonumber \\
        &\leq \int_0^{\bar\tau}\frac{1}{h(t)} \,dt \nonumber \\
          &\leq \frac{1}{\epsilon}\int_0^{\tau_{r(\epsilon)}} \sqrt{b'(k\bar x(t))} \,dt + \frac{1}{\epsilon}\int_{\tau_{r(\epsilon)}}^{\max\{\tau_{r(\epsilon)}, \bar\tau\}}\sqrt{b'(r(\epsilon))} \,dt \nonumber \\
          &\leq\frac{1}{\epsilon}\int_{x_0}^{\infty}\frac{\sqrt{b'(kx)}}{b(x)}\,dx + \frac{\abs{\bar\tau-\tau_{r(\epsilon)}}}{\epsilon}\sqrt{b'(r(\epsilon))} \nonumber \\
          &\stackrel{\ref{asmp:b.integrabilityCond}}{\le} \frac{C}{\epsilon} + \abs{\bar\tau-\tau_{r(\epsilon)}}\mathcal{O}(\epsilon^{\frac{\alpha}{2}-1}),
    \end{align}
    \endgroup
    where the transition from the second to the third line uses \eqref{eq:xBarBndX}, as this shows that \(\bar x(t) \leq x(t)\), allowing a change of variable in the left integral using \(dt = \frac{dx}{b(x)}\). 

    Inserting into~\eqref{eq:secondErrorTermFinal} gives
    \[\abs{\bar\tau-\tau_{r(\epsilon)}} \leq C\epsilon + \abs{\bar\tau-\tau_{r(\epsilon)}}\mathcal{O}(\epsilon^{1+\frac{\alpha}{2}}) \implies 
    \abs{\bar\tau-\tau_{r(\epsilon)}} 
    \leq \frac{C\epsilon}{(1-\mathcal{O}(\epsilon^{1+\frac{\alpha}{2}}))} 
    = C'\epsilon,\]
    as \(1+\frac{\alpha}{2}>0\), and thus
    \(N \stackrel{\eqref{eq:costEst1D}}{=} \mathcal{O}(\epsilon^{-1})\). 
    
    We then deduce from the above that 
    \[
    \abs{\bar\tau-\tau} \leq \abs{\bar\tau-\tau_{r(\epsilon)}} + \abs{\tau_{r(\epsilon)}-\tau} \leq C'\epsilon + \epsilon = \mathcal{O}(\epsilon).
    \]
    Furthermore, the computational cost is given by \(C''N \leq \frac{C'C''}{\epsilon} = \mathcal{O}(\epsilon^{-1})\), for some constant \(C''\in\N\).
\end{proof}

\begin{remark}[Alternative to Assumption \ref{asmp:b.FandG} revisited]
\label{rem:followAltAsmpThm1}
    We will here show that Theorem \ref{thm:error_alg1} holds under the alternative to Assumption \ref{asmp:b.FandG} given in Remark \ref{rem:altAsmpThm1}.
    Note that \(\frac{1}{b(x)} = \frac{d}{dx}\left(\frac{\log(b(x))}{b'(x)}\right) + \frac{\log(b(x))}{b'(x)^2}b''(x)\) and that integrating the first term of the right-hand side from \(r(\epsilon)\) to \(\infty\) yields the negative contribution \(-\frac{\log(b(r(\epsilon)))}{b'(r(\epsilon))}\).
    Then
    \[\abs{\tau_{r(\epsilon)} - \tau} \leq \int_{r(\epsilon)}^\infty \frac{\log(b(x))}{b'(x)^2}b''(x)\, dx \leq \int_{r(\epsilon)}^\infty \frac{(C+ 1 )\log(b'(x))}{b'(x)^2}b''(x)\, dx,\]
    where the second inequality follows as \(b(x) \leq xb'(x) \leq b'(x)^{C+1}\), using the alternative assumptions in Remark \ref{rem:altAsmpThm1}. 
    Integrating this expression gives \(\abs{\tau_{r(\epsilon)}-\tau}\leq\frac{(C+1)\log(b'(r(\epsilon)))}{b'(r(\epsilon))}\). And 
    choosing the threshold function $r(\epsilon)$ such that 
    \(b'(r(\epsilon)) = \epsilon^{-1}\log(\epsilon^{-1})\) suffices to ensure that \(\abs{\tau_{r(\epsilon)}-\tau} = \mathcal{O}(\epsilon)\).
    This yields \(b'(r(\epsilon)) = \mathcal{O}(\epsilon^\alpha)\) for some \(\alpha>-2\). 
    In the proof of Theorem~\ref{thm:error_alg1}, the above argument replaces the start of the proof, up to and including \eqref{eq:boundB_der_r}. The remaining part of the proof still applies under the alternative assumption. 
\end{remark}

\begin{remark}[Extension to separable ODEs]
\label{rem:seperableODEs}
The method can in principle be extended to stopping times of separable non-autonomous ODEs on the form
\begin{equation*}
\begin{split}
    x'(t) &= f(t, x(t)) = g(t)b(x(t)), \qquad t \in (0, \tau)\\
    x(0) &= x_0,
\end{split}
\end{equation*}
where \(g:C((0,\infty))\) is a positive and locally integrable function. 
The explosion time $\tau$ then is redefined through the implicit equation 
\(\int_0^\tau g(t)\,dt = \int_{x_0}^\infty \frac{1}{b(x)}\,dx\), or, equivalently, 
\[
\tau = G^{-1}\left(\int_{x_0}^\infty \frac{1}{b(x)}\,dx + G(0)\right),
\]
where \(G\) denotes an antiderivative of \(g\).
We can then calculate the first term in the argument of \(G^{-1}\) using Algorithm \ref{alg:1d}.
\end{remark}

\begin{remark}[Higher-order methods]
\label{rem:ho_sceme}
    Theorem \ref{thm:error_alg1} applies to the forward Euler numerical integrator, which is a first-order method. However, under the more restrictive constraints on $b$ given below, it is possible to improve the efficiency of the method further by using a higher-order Taylor scheme. 
    This higher-order method is for some integer \(2\leq\bar m \leq m\) given by \(\bar x_{n+1} = \bar x_n + \sum_{i=1}^{\bar m}\frac{x^{(i)}(\hat t)}{i!}h^i_{n}\), where \(\hat t\) is such that \(x(\hat t)= \bar x_n\). 
    The time stepping needs to be redefined as \(h_n = \frac{\epsilon^{\frac{1}{\bar m}}}{b'(\min\{k\bar x_n, r(\epsilon)\})^{\frac{\bar m}{1+\bar m}}}\). 

    The extension of Assumption~\ref{asmp:b} for the higher-order method are as follows: 
    Let \(b\in C^m(D)\) for \(m \geq 3\), let there exist some \(C>0\) such that \(\frac{d^m}{dx^m}b(x(t)) \leq Cb(x(t))b'(x(t))^m\), let~\ref{asmp:b.FandG} hold 
    for \(\alpha > -\frac{m+1}{m}\) and let there exist \(k>1\) such that \(\int_{x_0}^{\infty}\frac{b'(kx)^{\frac{\bar m}{\bar m + 1}}}{b(x)}\,dx < \infty\). Then one can show 
    that \(\abs{\tau - \bar\tau} = \mathcal{O}(\epsilon)\) is reached 
    at \(\mathcal{O}\left(\epsilon^{-\frac{1}{\bar m}}\right)\) computational cost. 
    For brevity, the proof of this result is omitted.
\end{remark}

\subsection{Comparing adaptive and uniform time steps}
\label{sec:adaptiveVsUniform1D}
Theorem \ref{thm:error_alg1} shows that our adaptive method reaches an error of order \(\mathcal{O}(\epsilon)\) with a cost of order \(\mathcal{O}(\epsilon^{-1})\). 
We will here find how much we gain, measured in computational cost, by using our adaptive algorithm instead of a naive approach using uniform steps. 
For brevity, we will assume that \(b\) satisfies Assumption~\ref{asmp:b} with \ref{asmp:b.FandG} replaced by the alternative condition in Remark~\ref{rem:altAsmpThm1}, as this provides a convenient bound on how \(b'(r)\) grows in $r$, see Remark \ref{rem:followAltAsmpThm1}. 
From \eqref{eq:secondTermExp} in the proof of Theorem \ref{thm:error_alg1} we have
\begin{align*}
    \abs{\bar\tau - \tau_{r(\epsilon)}} &\leq \sum_{n=1}^{N-1}b'(\bar x_n)h^2 + h\\
    &\leq h\int_{0}^{\tau_{r(\epsilon)}}b'(x(t))\,dt + h\int_{\tau_{r(\epsilon)}}^{\bar\tau}b'(r(\epsilon))\,dt + h \\
    &= h\log\left(\frac{b(r(\epsilon))}{b(x_0)}\right) + h\abs{\bar\tau - \tau_{r(\epsilon)}}b'(r(\epsilon)) + h,
\end{align*}
as \(\frac{d}{dt}\log(b(x(t))) = b'(x(t))\). Rearranging, we see that
\[\abs{\bar\tau - \tau_r} \leq \frac{h\log\left(\frac{b(r)}{b(x_0)}\right) + h}{1 - hb'(r)}\]
We want to pick \(h\) such that \(\abs{\bar\tau - \tau_r} = \mathcal{O}(\epsilon)\), this implies 
\[h = \min\left\{\bar h, \tilde h\right\} =\min\left\{\frac{\epsilon}{\log\left(\frac{b(r)}{b(x_0)}\right)}, \frac{1}{2b'(r)}\right\}.\]
The cost associated with \(\bar h\) is
\[N = \sum_{n=1}^N\frac{\bar h}{\bar h} \leq \frac{1}{\epsilon}\int_{0}^{\bar\tau}\log\left(\frac{b(r(\epsilon))}{b(x_0)}\right)\,dt = \frac{\bar\tau}{\epsilon}\log\left(\frac{b(r(\epsilon))}{b(x_0)}\right),\]
while we for \(\tilde h\) have
\[N =  \sum_{n=1}^N\frac{\tilde h}{\tilde h} \leq \int_0^{\bar\tau}2b'(r(\epsilon))\,dt = \frac{2\bar\tau}{\epsilon}\log(\epsilon^{-1}),\]
when using that \(b'(r(\epsilon)) = \mathcal{O}(\epsilon^{-1}\log(\epsilon^{-1}))\), 
cf.~Remark~ \ref{rem:followAltAsmpThm1}.
This shows that using uniform steps increase the asymptotic cost with at least a factor 
$\log(\epsilon^{-1})$ relative to the adaptive method. 
We will explore this numerically for different examples in Section~\ref{sec:numExp}.  

\section{Extension to higher dimensions}
\label{sec:higherDim}

This section introduces the higher-dimensional analogue of the adaptive method in Section \ref{sec:adaptive1D}. 
The asymptotic rates for the approximation error and computational cost is provided in Theorem \ref{thm:propertiesAlgRn}. 
Thereafter we explore stronger, but often verifiable, regularity conditions on the vector field \(b\), cf. Lemmas \ref{lemma:SuffCondAsmpGradTau} and \ref{lemma:suffCondAsmpbDerOverBBound}.
Finally, in Section \ref{sec:RnLogDyn} we extend the method to slower growing vector fields, such as \(\abs{b(x)} = \abs{x}\log(\abs{x})\).

We now consider the higher-dimensional ODE 
\begin{equation}
\label{eq:systemODE}
    \begin{split}
        \dot{x} &= b(x) \\
        x(0) &= x_0 \in D,
    \end{split}
\end{equation}
where $D \coloneq \{x \in \R^n \mid |x| >\delta\}$ for 
a given $\delta >0$ and where \(\abs{\cdot}\) denotes the Euclidean norm on \(\R^n\). Conditions on $b:D \to \R^n$ that ensure existence of a unique $C^1$-solution
that blows up in finite time are detailed in Assumption~\ref{asmp:Rn} below. Our objective is once again to compute the blow up time 
\begin{equation*}
\tau(x_0) = \inf \left\{t \geq 0 \,\middle\vert\, \lim_{s \uparrow t} |x(s)| = \infty, x(0) = x_0\right\}.
\end{equation*}

Before presenting assumptions on \(b\), we introduce some notation. 
Let \(\norm{\cdot}\) denote the induced matrix 2-norm,  
$B(\xi) \coloneqq \{x \mid |x| < \xi\}$ and \(b': D \to \R^{n \times n}\) will be used to denote the Jacobian of \(b\). Furthermore, let \(r:(0,\infty) \to (\delta,\infty)\) denote a problem-dependent, monotone-decreasing and continuous function 
satisfying \(\underset{\epsilon\to 0}{\lim}r(\epsilon) = \infty\).
For $\epsilon >0$, consider the subset $D(\epsilon) \coloneqq D \cap B(r(\epsilon)) \subset D$
and let
\[
\tau_\epsilon(x_0) \coloneqq \underset{t\geq0}{\min}\{ t \ge 0 \mid |x(t)| \geq r(\epsilon), \; x(0) = x_0\} 
\]
denote the solution's exit-time of the set $D(\epsilon)$.

\begin{assumption}
\label{asmp:Rn}
For the ODE~\eqref{eq:systemODE}, the function \(b\in C^1(D, \R^n)\) 
satisfies the following conditions:
        \begin{enumerate}[label=\arabic*), ref=(\theassumption.\arabic*)]
        \item There exist \(\check{C}, \alpha > 0\) such that \(\check C \abs{x}^{2+\alpha} \leq b(x)\cdot x\) for all \(x \in D\).\label{asmp:Rn.lwrAndUprBndGrowthSqrdNorm}
        
        \item There exist \(C_1 > 0\) such that for any \(\epsilon>0\) and all \(x_0 \in D(\epsilon)\), we have that \(\abs{\nabla \tau_\epsilon(x_0)} \leq \frac{C_1}{\abs{b(x_0)}}\). \label{asmp:Rn.BndGradTau}
        
        \item At least one of the following holds:
        \begin{enumerate}[label=\alph*), ref=(\theassumption.3\alph*)]
            \item There is a constant \(C_2 \in (0,1]\) such that \(C_2 \leq \frac{b(x)\cdot x}{\abs{b(x)}\abs{x}}\) holds for all \(x\in D\).\label{asmp:Rn.directionalityOfB.ineq}
            \item For all \(x\in D\), we have \(x_ib_i(x)\geq0\) for \(i \in \{1, 2, \hdots, n\}\).\label{asmp:Rn.directionalityOfB.sign}
        \end{enumerate}
        \label{asmp:Rn.directionalityOfB}
        
        \item We can find constants \(C_3, C_4 >0,\, \nu>1,\, \upsilon \geq \max\{1-\alpha, 0\}\) such that for \(h>0\) sufficiently small, we have \(\frac{\sqrt{\norm{b'(x)}}}{\abs{b(x+b(x)h)}} \leq \frac{C_3}{\abs{x + b(x)h}^\nu}\) and \(\frac{\norm{b'(x)}}{\abs{b(x)}} \leq \frac{C_4}{\abs{x}^\upsilon}\) for all \(x \in D\).
        \label{asmp:Rn.bDerOverBBound}
        \item Given \(h>0\) sufficiently small, we have that \(\abs{b(x)} \leq \abs{b(x + b(x)h)}\) for any \(x \in D\). \label{asmp:Rn.bndNormBDirection}
        \item There is a constant \(k>1\) such that for any \(\epsilon>0, \theta\in(0,1)\) and \( x\in D(\epsilon)\), it holds that
        \[\norm{b'\left(x + \frac{\epsilon}{\sqrt{\norm{b'(x)}}}\theta b(x)\right)} \leq k\norm{b'(x)}.\] \label{asmp:Rn.bndBdNextByPrev}
    \end{enumerate}
\end{assumption}
We note that~\ref{asmp:Rn.lwrAndUprBndGrowthSqrdNorm} implies finite-time blow up of any solution of the ODE as $d |x|^2/dt > 2\check C |x|^{2+\alpha}$, and $b \in C^1(D)$ ensures there exists a unique solution $x\in C^1([0,\tau), \R^n)$.

\begin{remark}[Monotonicity of \(\abs{\bar x_n}\)]
    Note that 
    \[\abs{\bar x_{n}}^2 = \abs{\bar x_{n-1} + b(\bar x_{n-1})h_{n-1}}^2 = \abs{\bar x_{n-1}}^2 + 2\bar x_{n-1} \cdot b(\bar x_{n-1})+ \abs{b(\bar x_{n-1})}^2h_{n-1}^2,\]
    then a consequence of Assumption \ref{asmp:Rn.directionalityOfB} is that
    \(\abs{\bar x_{n-1}}\leq\abs{\bar x_{n}}\).
\end{remark}

The multidimensional version of our method is given in Algorithm \ref{alg:Rn}. 
The method estimates \(\tau\) by approximating \(\tau_\epsilon\), where \(r(\epsilon)\) is set such that \(\abs{\tau-\tau_\epsilon} = \mathcal{O}(\epsilon)\), see Lemma \ref{lemma:boundTauDiff}. 
This is done by solving the system of ODEs using forward Euler with the adaptive step length \(h_n = \frac{\epsilon}{\sqrt{\norm{b'(\bar x)}}}\) until the norm of the solution first passes \(r(\epsilon)\). 
The current time of the solver, \(t_N\), is then returned as our estimate \(\bar\tau\). The adaptive step lengths are chosen to ensure \(\abs{\bar\tau - \tau_\epsilon} = \mathcal{O}(\epsilon)\), see proof of Theorem \ref{thm:propertiesAlgRn}.
\begin{algorithm}
    \caption{Estimating blow-up time of ODE in \(\R^n\)}
    \label{alg:Rn}
    \begin{algorithmic}[1]
        \Require \(b,\, x_0,\,\alpha,\, \check{C},\,\epsilon > 0\).
        \State Set \(t = 0,\, N=0,\,  \bar x = x_0\) and \(r(\epsilon) = \left(\frac{1}{\check{C}\alpha\epsilon}\right)^{\frac{1}{\alpha}}\).
        \While{\(\abs{\bar x} \leq r(\epsilon)\)}
            \State \(h = \frac{\epsilon}{\sqrt{\max\{\norm{b'(\bar x)}, 1\}}}\)
            \State \(\bar x \leftarrow \bar x + b(\bar x)h\).
            \State \(t \leftarrow t + h\).
            \State \(N \leftarrow N+1\).
        \EndWhile
        \State \(\bar\tau = t\).
        \Ensure \(\bar\tau,\,N\).
    \end{algorithmic}
\end{algorithm}

Note that Algorithm \ref{alg:Rn} is very similar to Algorithm \ref{alg:1d}, with the main differences being how \(r(\epsilon)\) is calculated and that the \(\R^n\)-version does not need a factor \(k\) in the argument of \(b'\). However, the assumptions on \(b\) are stricter in the \(\R^n\) case compared to the 1D case.

From our definition of \(D\) and \(D(\epsilon)\) we have a disjoint partition \(\partial D(\epsilon) = \partial D_1(\epsilon)\sqcup\partial D_2(\epsilon)\), where \(\partial D_1(\epsilon)\) is the surface of the ball \(B(r(\epsilon))\), while \(\partial D_2(\epsilon)\) is the surface of \(B(r(\delta))\). 
Note from Assumption \ref{asmp:Rn} that if \(x_0\in D(\epsilon)\), the process \(x(t)\) will never hit the inner boundary \(\partial D_2(\epsilon)\) as \(x\cdot b(x) > 0\) for all \(x \in D\). 
In addition, we have that \(\tau_\epsilon\) satisfies the Hamilton-Jacobi PDE 
\begin{equation}
\label{eq:PDE_exit_time}
\begin{split}
    \nabla\tau_\epsilon(x)\cdot b(x) &= -1,\quad x\in D(\epsilon) \\
    \tau_\epsilon(x) &= 0,\hspace{6.25mm} x\in\partial D_1(\epsilon),
\end{split}
\end{equation}
giving a characterization for \(\partial D_1(\epsilon)\). Note that the characteristics of \(\tau_\epsilon\) are simply the flows \(x(t)\) solving \eqref{eq:systemODE}. Due to our assumptions on \(b\), we have that these flows are well-behaved in such a way that the unique solution of \eqref{eq:PDE_exit_time} is \(\tau_\epsilon \in C^1(\overline{D(\epsilon)})\), cf.~\cite[Chapter 3]{Evans2022}.

Before proceeding to the main result of this section, we need the following lemma providing an expression for \(r(\epsilon)\) that guarantees that we can control the error \(\abs{\tau(x_0)-\tau_\epsilon(x_0)}\).
\begin{lemma}
\label{lemma:boundTauDiff}
Let \(r(\epsilon) = \left(\frac{1}{\check{C}\alpha\epsilon}\right)^{\frac{1}{\alpha}}\) for \(\epsilon > 0\).
Under Assumption \ref{asmp:Rn.lwrAndUprBndGrowthSqrdNorm} it then holds for all $x_0 \in D$ that 
\[
\abs{\tau(x_0)-\tau_\epsilon(x_0)} \leq \epsilon \quad \text{and} \quad 
\tau(x_0) \leq \frac{1}{\check{C}\alpha\abs{x_0}^\alpha}.
\]
\end{lemma}
\begin{proof}
Assumption \ref{asmp:Rn.lwrAndUprBndGrowthSqrdNorm} yields that 
\( \frac{d\abs{x(t)}^2}{dt} \ge 2\check{C}\abs{x(t)}^{2+\alpha}\). 
Introducing \(y = \abs{x}^2\), we obtain that
\[\tau(x_0)-\tau_\epsilon(x_0) = \int_{\tau_\epsilon(x_0)}^{\tau(x_0)}\,dt \leq \frac{1}{2\check{C}}\int_{r(\epsilon)^2}^{\infty} \frac{1}{y^{1+\frac{\alpha}{2}}}\,dy 
= \frac{1}{\check{C}\alpha r(\epsilon)^\alpha} = \epsilon,\]
and
\[
\tau(x_0) = \int^{\tau(x_0)}_{0} \, dt = \frac{1}{\check{C}\alpha\abs{x_0}^\alpha}\, .
\]
\end{proof}

We can then proceed to our main finding. 
\begin{theorem}[Properties of Algorithm \ref{alg:Rn}]
\label{thm:propertiesAlgRn}
    Let Assumption \ref{asmp:Rn} hold and \(\bar\tau\) be calculated using Algorithm \ref{alg:Rn}. 
    Then, for any \(x_0 \in D\) and any sufficiently small \(\epsilon>0\), it holds that
    \begin{equation*}
    \abs{\bar\tau(x_0)-\tau(x_0)} = \mathcal{O}\left(\epsilon\right)
    \quad \text{and} \quad \mathrm{Cost}(\bar \tau) = \mathcal{O}(\epsilon^{-1}). 
\end{equation*} 
\end{theorem}
\begin{proof}
We begin by considering the difference \(\abs{\bar\tau(x_0)-\tau_\epsilon(x_0)}\). Introduce the notation 
\[\tau_\epsilon(x_0, t_0) = t_0 + \underset{t \geq 0}{\min}\{x(t) \geq D(\epsilon) \mid x(0) = x_0\}.\]
We are then able to write 
\begingroup 
\allowdisplaybreaks
\begin{align}
\label{eq:boundHD}
    \abs{\bar\tau(x_0)-\tau_\epsilon(x_0)} &= \abs{\tau_\epsilon(\bar x_N, t_N)-\tau_\epsilon(x_0, t_0)} \nonumber \\
                          &\leq \sum_{n=1}^{N}\abs{\tau_\epsilon(\bar x_n, t_n) - \tau_\epsilon(\bar x_{n-1}, t_{n-1})} \nonumber\\
                          &= \sum_{n=1}^{N}\abs{\tau_\epsilon(\bar x_n, t_n) - \tau_\epsilon(\bar x_{n-1}, t_{n}) + h_{n-1}} \nonumber\\
                          &= \sum_{n=1}^{N}\bigg\vert \nabla\tau_\epsilon(\hat x_{n-1})\cdot(\bar x_n - \bar x_{n-1}) + h_{n-1} \bigg\vert \nonumber\\
                          &= \sum_{n=1}^{N}\bigg\vert\nabla\tau_\epsilon(\hat x_{n-1})\cdot b(\bar x_{n-1})h_{n-1} + h_{n-1} \nonumber \\
                          &\hspace{10mm}+\nabla\tau_\epsilon(\hat x_{n-1})\cdot b(\hat x_{n-1})h_{n-1} - \nabla\tau_\epsilon(\hat x_{n-1})\cdot b(\hat x_{n-1})h_{n-1}\bigg\vert \nonumber\\
                          &\overset{\eqref{eq:PDE_exit_time}}{=} \sum_{n=1}^{N}\abs{\nabla\tau_\epsilon(\hat x_{n-1})\cdot\left(b(\bar x_{n-1})-b(\hat x_{n-1})\right)h_{n-1}} \nonumber \\
                          &\leq \sum_{n=1}^{N}\abs{\nabla\tau_\epsilon(\hat x_{n-1})\cdot \left(b'\left(\hat{\hat{x}}_{n-1}\right)b(\bar x_{n-1})\right)h_{n-1}^2} \nonumber \\
                          \overset{\ref{asmp:Rn.bndBdNextByPrev}}&{\leq} k\sum_{n=1}^{N}\abs{\nabla\tau_\epsilon(\hat x_{n-1})}\norm{b'\left(\bar x_{n-1}\right)}\abs{b(\bar x_{n-1})}h_{n-1}^2 \nonumber\\
                          \overset{\text{\ref{asmp:Rn.BndGradTau},\ref{asmp:Rn.bndNormBDirection}}} &{\leq} C_1k\sum_{n=1}^{N}\frac{\norm{b'\left(\bar x_{n-1}\right)}\abs{b(\bar x_{n-1})}}{\abs{b(\bar x_{n-1})}}h_{n-1}^2 \\
                          &\leq C_1k\sum_{n=1}^{N}\norm{b'(\bar x_{n-1})}h_{n-1}^2\nonumber,
\end{align}
with \(\hat x_{n-1} = \bar x_{n-1} + \theta b(\bar x_{n-1})h_{n-1}\) and \(\hat{\hat{x}}_{n-1} = \bar x_{n-1} + \theta' b(\bar x_{n-1})h_{n-1}\) for some \(\theta \in (0,1)\) and \(\theta'\in (0, \theta)\). From Algorithm \ref{alg:Rn}, we have that 
    $h_{n-1} = \epsilon/\sqrt{\norm{b'(\bar x_{n-1})}}$, and inserting this 
    into the above inequality gives 
\begin{equation}
    \label{eq:error_bnd_Rn_last}
\abs{\bar\tau(x_0)-\tau_\epsilon(x_0)} \leq kC\sum_{n=1}^N\epsilon^2 = kCN\epsilon^2.
\end{equation}
\endgroup
Lemma \ref{lemma:LwrBndDeltaNormX} below implies that \(\Delta\abs{\bar x}_{n-1} \coloneqq \abs{\bar x_{n}}_1-\abs{\bar x_{n-1}}_1 \geq C\abs{b(\bar x_{n-1})}_1h_{n-1}\) for some constant \(C\in(0,1]\) and Assumption \ref{asmp:Rn.lwrAndUprBndGrowthSqrdNorm} 
implies that \(\abs{b(x)}_1 \geq C'\abs{x}_1^{1+\alpha}\) for some  \( C'>0\), where \(\abs{\cdot}_1\) denotes the 1-norm. 
Thanks to these properties, we obtain that
\begin{align*}
    N - 1 &= \sum_{n=1}^{N-1} \frac{\abs{b(\bar x_{n})}_1}{\abs{b(\bar x_{n})}_1h_{n-1}}h_{n-1} \\
    &= \sum_{n=1}^{N-1} \frac{\abs{b(\bar x_{n})}_1 + \abs{b(\bar x_{n-1})}_1 - \abs{b(\bar x_{n-1})}_1}{\abs{b(\bar x_{n})}_1h_{n-1}}h_{n-1} \\
          &\leq C\sum_{n=1}^{N-1} \left(\frac{1}{\abs{b(\bar x_{n})}_1h_{n-1}}\Delta\abs{\bar x}_{n-1} + \frac{\norm{b'(\bar x_{n})}_1}{\abs{b(\bar x_{n})}_1}\Delta\abs{\bar x}_{n-1}\right) \\
          &\leq C\sum_{n=1}^{N-1} \left(\frac{\sqrt{\norm{b'(\bar x_{n-1})}}_1}{\epsilon\abs{b(\bar x_n)}_1}\Delta\abs{\bar x}_{n-1} + \frac{\norm{b'(\bar x_{n})}_1}{\abs{b(\bar x_{n})}_1}\Delta\abs{\bar x}_{n-1}\right) \\
          \overset{\ref{asmp:Rn.bDerOverBBound}}&{\leq} C\sum_{n=1}^{N-1} \left(\epsilon^{-1}\abs{\bar x_n}_1^{-\nu}\Delta\abs{\bar x}_{n-1} + \abs{\bar x_n}_1^{-\upsilon}\Delta\abs{\bar x}_{n-1}\right) \\
          &\leq C\epsilon^{-1}\left[-\bar x^{1-\nu}\right]_{\abs{x_0}}^{r(\epsilon)} + \mathcal{O}\left(\epsilon^{-1}\right)\\
          & =\mathcal{O}\left(\epsilon^{-1}\right),
\end{align*}
where \(C\) is a positive constant that absorb other constants.
Inserting this bound for \(N\) in \eqref{eq:error_bnd_Rn_last} yields 
$\abs{\bar\tau(x_0)-\tau_{\epsilon}(x_0)} = \mathcal{O}\left(\epsilon\right)$
and 
\begin{align*}
    &\abs{\bar\tau(x_0)-\tau(x_0)} \leq \abs{\bar\tau(x_0)-\tau_\epsilon(x_0)} + \abs{\tau(x_0)-\tau_\epsilon(x_0)} 
    \stackrel{\text{Lemma \ref{lemma:boundTauDiff}}}{=} \mathcal{O}\left(\epsilon\right).
\end{align*}
\end{proof}
\begin{remark}[Alternative choice for \(h\)]
\label{rem:altHRn}
    In some situations, being able to set \(h_{n-1} = \epsilon\sqrt{\abs{b(\bar x_{n-1})}}/\sqrt{\abs{b'(\bar x_{n-1})b(\bar x_{n-1})}}\) could be beneficial if either the value of \(\sqrt{\norm{b'(x)}}\) is significantly larger than \(\sqrt{\abs{b'(x)b(x)}}\) or if the former is more expensive to compute than the latter. From the calculations leading to \eqref{eq:boundHD}, we see that this change of adaptive step length is possible provided \(\abs{b'(x+b(x)h)b(x)} \leq C\abs{b'(x)b(x)}\) holds for some \(C>0\). Additionally, one needs to adjust the first part of Assumption \ref{asmp:Rn.bDerOverBBound} accordingly for the result of Theorem \ref{thm:propertiesAlgRn} to still hold. 
\end{remark}

The next lemma is used in the proof of Theorem \ref{thm:propertiesAlgRn} above.

\begin{lemma}[Lower bound on \(\Delta \abs{\bar x}_{n-1}\)]
\label{lemma:LwrBndDeltaNormX}
If Assumption~\ref{asmp:Rn.directionalityOfB} holds, then~\\\mbox{\(\Delta \abs{\bar x}_{n-1} \coloneqq \abs{\bar x_{n}}_1 - \abs{\bar x_{n-1}}_1\geq C\abs{b(\bar x_{n-1})}_1h_{n-1}\)} for a constant \(C \in (0, 1]\).
\end{lemma}

\begin{proof}
    Let us first consider the case where Assumption \ref{asmp:Rn.directionalityOfB.ineq} holds. We then have that
    \begin{align*}
        \abs{\bar x_{n}} &= \abs{\bar x_{n-1} + b(\bar x_{n-1})h_{n-1}} \\
        &\geq \abs{(\bar x_{n-1} + b(\bar x_{n-1})h_{n-1})\cdot\frac{\bar x_{n-1}}{\abs{\bar x_{n-1}}^2}\bar x_{n-1}}\\
        &\geq \abs{\left(1 + \frac{b(\bar x_{n-1})\cdot \bar x_{n-1}}{\abs{\bar x_{n-1}}^2}h_{n-1}\right)\bar x_{n-1}} \geq \abs{\bar x_{n-1}} + C_2\abs{b(\bar x_{n-1})}h_{n-1},
    \end{align*}
    and by equivalence of norms this implies our result. Assume now that Assumption \ref{asmp:Rn.directionalityOfB.sign} holds. Then the claim holds as
    \[\abs{\bar x_n}_1 - \abs{\bar x_{n-1}}_1 = \abs{\bar x_{n-1} + b(\bar x_{n-1})h_{n-1}}_1 - \abs{\bar x_{n-1}}_1 = \abs{b(\bar x_{n-1})}_1 h_{n-1}\]
\end{proof}

Some points in Assumption~\ref{asmp:Rn} can be difficult to verify for a given vector 
field. 
For example, verifying Assumption~\ref{asmp:Rn.BndGradTau} analytically requires knowledge of the form of the gradient \(\nabla\tau_\epsilon(x_0)\) which often will not be available. Lemma~\ref{lemma:SuffCondAsmpGradTau} introduces a sufficient condition for Assumption~\ref{asmp:Rn.BndGradTau} that depends on bounding the first variation of \(x\). We let \(J\) denote the first variation of \(x\), that is
    \[J(t, x_0) \coloneq D_{x_0} x(t, x_0)\]
    with \(D_{x_0}\) denoting the Jacobian with respect to \(x_0\)
    
\begin{lemma}
\label{lemma:SuffCondAsmpGradTau}
    Let Assumption \ref{asmp:Rn}, with the exception of Assumption \ref{asmp:Rn.BndGradTau}, hold and assume there exist \(C > 0\) such that for any $\epsilon>0$ and all $x_0\in D(\epsilon)$, 
    \[
    \norm{J(\tau_\epsilon(x_0), x_0)} \leq C\frac{\abs{b(x(\tau_\epsilon(x_0), x_0))}}{\abs{b(x_0)}}. 
    \]
    Then Assumption \ref{asmp:Rn.BndGradTau} holds as well. 
\end{lemma}
\begin{proof}
    The domain $D(\epsilon) \subset B(r(\epsilon))$, and since $|x(t)|$ is strictly increasing, we recall that the solution will hit the sphere $\partial B(r(\epsilon))$ at 
    $t = \tau_\epsilon(x_0)$. The function \(\phi(x) = \abs{x}^2-r(\epsilon)^2\) 
    is equal to zero on $\partial B(r(\epsilon))$. Writing, 
    $f(x_0) = \phi(x(\tau_\epsilon(x_0), x_0))$, we thus obtain 
    \[
    \frac{\partial}{\partial x_0^i}f(x) = 
    0 \qquad i =1,\ldots, n,
    \]
    for any $x_0 = (x_0^1,\ldots, x_0^n) \in D(\epsilon)$. Expanding the left side of the above equation yields 
    \[\nabla f(x_0)^T\left( b(x(\tau_\epsilon(x_0), x_0)) \partial _{x_0^i}\tau_\epsilon(x_0) + x_{x_0^i}(\tau_\epsilon(x_0), x_0)\right) = 0, \]
    where \(x_{x_0^i}(\tau_\epsilon(x_0), x_0)\) is the \(i\)-th column of \(J(\tau_\epsilon(x_0), x_0)\)
    and thus
    \begin{equation}
    \label{eq:gradTauDir}
        \nabla\tau_\epsilon(x_0) = -\frac{x(\tau_\epsilon(x_0), x_0)^T J(\tau_\epsilon(x_0), x_0)}{x(\tau_\epsilon(x_0), x_0)^T  b(x(\tau_\epsilon(x_0), x_0))}.
    \end{equation}
    Hence 
    \begin{equation*}
    \abs{\nabla\tau_\epsilon(x_0)} \stackrel{\ref{asmp:Rn.directionalityOfB.ineq}}{\leq} \frac{\norm{J(\tau_\epsilon(x_0), x_0)}}    {C_2\abs{b(x(\tau_\epsilon(x_0), x_0))}} \leq \frac{C\abs{b(x(\tau_\epsilon(x_0), x_0))}}{C_2\abs{b(x(\tau_\epsilon(x_0), x_0))}\abs{b(x_0)}} = \frac{C_1}{\abs{b(x_0)}},
    \end{equation*}
    where \(C_1 = \frac{C}{C_2}\) is some positive constant. 
\end{proof}

Often an expression for \(J\) will not be readily available, possibly making the bound needed in Lemma \ref{lemma:SuffCondAsmpGradTau} hard to find.
An alternative is then to use the sufficient condition given in Remark \ref{rem:suffCondAsmpFirstVar} related to quadratic forms of \(b'(x)\).

\begin{remark}
   \label{rem:suffCondAsmpFirstVar}
   Let \(v_{k}(t)\) denote the \(k\)-th column of the first variation \(J(t,x_0)\).
   Then \(\dot{v_k} = b'(x)v_k\) with \(v_k(0) = e_k\) as \(J(0,x_0) = I\). 
   A sufficient condition is that there exist some constant \(C > 0\) such that
   \begin{equation}
          \label{eq:suffCondAssmpFirstVar}
          \underset{v\in \bar B\left(\abs{b(x)}\right)}{\sup} v^Tb'(x)v = \frac{1}{2}\lambda_{\max}\left(b'(x)+b'(x)^T\right)\abs{b(x)}^2\leq C^2b(x)^Tb'(x)b(x).
   \end{equation}
   In addition to \(C\abs{b(x_0)}\geq 1\) and \(b(x_0)b'(x_0)b(x_0) - \frac{\lambda_{\max}(b'(x_0)+b'(x_0)^T)}{2C^2}\geq 0\). 
   This is the case as then 
   \[\frac{d}{dt}\left(\frac{\abs{b(x(t))}^2}{2} - \frac{\abs{v_k(t)}^2}{2C^2}\right) = b(x(t))b'(x(t))b(x(t)) - \frac{v_k^T(t)b'(x(t))v_k^T(t)}{C^2} \geq 0,\] 
   implying that \(\abs{v_k(t)} \leq C\frac{\abs{b(x(t))}}{\abs{b(x_0)}}\), as needed. 
\end{remark}

Lastly, we also provide a convenient, easy-to-check, sufficient condition for Assumption \ref{asmp:Rn.bDerOverBBound}.

\begin{lemma}
\label{lemma:suffCondAsmpbDerOverBBound}
Assume that Assumption \ref{asmp:Rn.lwrAndUprBndGrowthSqrdNorm} holds and that there exist constants \(C>0\) and \(0<\gamma <\min\{2\alpha, 1+\alpha\}\) such that \(\norm{b'(x)}\leq C_3\abs{x}^{\gamma}\) for all \(x \in D\). Then Assumption \ref{asmp:Rn.bDerOverBBound} holds with \(\nu = -\frac{\gamma-2(1+\alpha)}{2}\) and \(\upsilon = -(\gamma-\alpha-1)\).
\end{lemma}
\begin{proof}
    We have that
    \[\frac{\sqrt{\norm{b'(x)}}}{\abs{b(x+b(x)h)}} \leq \frac{\sqrt{\abs{x+b(x)h}^{\gamma}}}{\abs{b(x+b(x)h)}} \leq \abs{x+b(x)h}^{\frac{\gamma-2(1+\alpha)}{2}},\]
    and as \(\gamma < 2\alpha\) we get \(\nu=-\frac{\gamma-2(1+\alpha)}{2}>1\). Furthermore, \(\frac{\norm{b'(x)}}{\abs{b(x)}} \leq \abs{x}^{\gamma-\alpha-1}\), and from the assumed restriction on \(\gamma\), we have \(\upsilon = -(\gamma-\alpha-1) \geq\max\{1-\alpha,0\}\).
\end{proof}

\subsection{Dynamics with slower growth}
\label{sec:RnLogDyn}

The theory above does not cover cases where \(\abs{b(x)}\) grows slower than \(\abs{x}^{1+\alpha}\) for \(\alpha >0\). Let us therefore consider an alternative to Assumption \ref{asmp:Rn.lwrAndUprBndGrowthSqrdNorm}, allowing for slower growing functions, such as \(\abs{b(x)} = \abs{x}\log(\abs{x})^{2}\). 

\begingroup
\renewcommand{\theassumption}{B\textsuperscript{\ensuremath\prime}}
\begin{assumption}
    \label{asmp:RnAlt}
    Consider the ODE~\eqref{eq:systemODE} with $D = \{x \in \R^n \mid |x| > \delta\}$ for a \(\delta > 1\). Let Assumption \ref{asmp:Rn} hold, with the possible exception of points \ref{asmp:Rn.lwrAndUprBndGrowthSqrdNorm} and \ref{asmp:Rn.bDerOverBBound}.
    And let the function \(b\in C^1(D)\) satisfy the following additional conditions:
    \begin{enumerate}[label=\arabic*)]
        \item There exists \(\check{C}, \alpha > 0\) such that \(\check{C}\abs{x}^2\log(\abs{x})^{1+\alpha} \leq b(x)\cdot x\) for all \(x \in D\). \\
        \stepcounter{enumi}
        \stepcounter{enumi}
        \item We can find constants \(C_3, \gamma >0\) such that \(\norm{b'(x)} \leq C_3\log(\abs{x})^{\gamma}\) for all \(x \in D\).
    \end{enumerate}
\end{assumption}
\addtocounter{assumption}{-1}
\endgroup

We can then find an analogue of Lemma \ref{lemma:boundTauDiff} under Assumption \ref{asmp:RnAlt}.
\begin{corr}
    \label{corr:boundTauDiffAlt}
    Let \(r(\epsilon) = e^{\left(\frac{1}{\check{C}\alpha\epsilon}\right)^{\frac{1}{\alpha}}}\) for \(\epsilon > 0\).
    Under Assumption \ref{asmp:RnAlt} it then holds that
    \[\abs{\tau(x_0)-\tau_\epsilon(x_0)} \leq \epsilon.\]
\end{corr}
\begin{proof}
    We have \(\frac{d\abs{x(t)}^2}{dt} = 2b(x(t))\cdot x(t)\). Assumption \ref{asmp:RnAlt} then gives \(2\check{C}\abs{x(t)}^2\log(\abs{x(t)})^{1+\alpha} \leq \frac{d\abs{x(t)}^2}{dt}\). Now, setting \(y = \abs{x}^2\) gives
    \[\abs{\tau(x_0)-\tau_\epsilon(x_0)} = \int^{\tau(x_0)}_{\tau_\epsilon(x_0)}\,dt \leq \frac{2^{\alpha}}{\check{C}}\int^\infty_{r(\epsilon)^2}\frac{1}{y\log(y)^{1+\alpha}}\,dy.\]
    Inserting the assumed value for \(r(\epsilon)\) yields
    \[\abs{\tau(x_0)-\tau_\epsilon(x_0)} \leq \frac{1}{\check{C}\alpha}\frac{1}{\log(r(\epsilon))^\alpha} = \epsilon.\]
\end{proof}
Finally, we can prove a result analogue to Theorem \ref{thm:propertiesAlgRn} for this setting. 
\begin{corr}
    \label{corr:propertiesAlgRnLog}
    Let Assumption \ref{asmp:RnAlt} hold with parameters such that \(\gamma \geq 1 + \alpha\). Moreover, let \(\bar\tau\) be calculated using Algorithm \ref{alg:Rn} with \(h_{n-1} = \sqrt{\frac{\epsilon }{N\max\{1, b'(\bar x_{n-1})\}}}\) and \(r(\epsilon)=e^{\left(\frac{1}{\check{C}\alpha\epsilon}\right)^{\frac{1}{\alpha}}}\). Then, for sufficiently small \(\epsilon\), it holds that
    \[\abs{\bar\tau(x_0) - \tau(x_0)} = \mathcal{O}(\epsilon)\]
    for all \(x_0 \in D\) with a computational cost of 
    \[Cost =  
    \begin{cases}
        \mathcal{O}(\epsilon^{-1}\log(\epsilon^{-1})^2) & \text{if } \frac{\gamma}{2} = \alpha \\
        \mathcal{O}(\epsilon^{1-\frac{\gamma}{\alpha}} + \epsilon^{-1}) & \text{otherwise}. \\
    \end{cases}\]
\end{corr}
\begin{proof}
    The error indicators are in this setting the same as what we have in \eqref{eq:boundHD}. We set \(h_{n-1} = \sqrt{\frac{\epsilon }{N\max\{1, b'(\bar x_{n-1})\}}}\), which gives an error of order \(\mathcal{O}(\epsilon)\). It then remains to find the cost. We have
    \begin{align*}
        N -1 &= \sum_{n = 1}^{N-1} \frac{1}{\abs{b(\bar x_{n})}h_{n-1}}\abs{b(\bar x_{n})}h_{n-1} \\
        &\leq \sum_{n=1}^{N-1} \left(\frac{\sqrt{\norm{b'(\bar x_{n-1})}}}{\epsilon\abs{b(\bar x_n)}} + C'\frac{\norm{b'(\bar x_{n})}}{\abs{b(\bar x_{n})}}\right)\Delta\abs{\bar x}_{n-1} \\
        &\leq \sum_{n=1}^{N-1} \left(C\sqrt{\frac{N}{\epsilon}}\log(\abs{\bar x_{n}})^{\frac{\gamma}{2}-1-\alpha} + C'\log(\abs{\bar x_{n}})^{\gamma-1-\alpha}\right)\abs{\bar x_{n}}^{-1}\Delta\abs{\bar x}_{n-1}. 
    \end{align*}
    The value of \(N\) now depends on the value of \(\alpha\) and \(\gamma\). We see that by dividing both sides by \(\sqrt{N}\) and bounding the sum by an integral we get
    \[\sqrt{N} \leq 
    \begin{cases}
        C\epsilon^{-\frac{1}{2}}\log(\epsilon^{-1}) + C'\epsilon^{-1}N^{-\frac{1}{2}} & \text{if } \frac{\gamma}{2} = \alpha \\
        C\epsilon^{\frac{1}{2} - \frac{\gamma}{2\alpha}} + C\epsilon^{-\frac{1}{2}} + C'\epsilon^{1 - \frac{\gamma}{\alpha}}N^{-\frac{1}{2}} & \text{otherwise}.
    \end{cases}\]
    By taking the square of the terms in these expressions we conclude that 
    \[N = 
    \begin{cases}
        \mathcal{O}(\epsilon^{-1}\log(\epsilon^{-1})^2) & \text{if } \frac{\gamma}{2} = \alpha \\
        \mathcal{O}(\epsilon^{1-\frac{\gamma}{\alpha}} + \epsilon^{-1}) & \text{otherwise}. \\
    \end{cases}\]
\end{proof}

\section{Numerical experiments}
\label{sec:numExp}
In this section we present a collection of numerical experiments where we test our a priori adaptive methods and compare them with alternative approaches.
Section \ref{sec:numExp_UniformVsAdaptive} and \ref{sec:numExp_1DNoAnaSolution} consider \(b(x) = x^2\) and \(b(x) = e^{x^2}\) respectively.  
Section \ref{sec:numExp_xlogx_c} studies the slower-growing function \(b(x) = x\log(x)^{1+c}\), an example not covered by our theoretical framework, showing how we can adapt in order to still get theoretical rates. 
Section \ref{sec:numExp_UncoupledProblem} and \ref{sec:numExp_CoupledProblem} then considers examples in \(\R^n\) where the components are uncoupled in the former and coupled in the latter, while Section \ref{sec:numExp_slowGrowthMultidim} looks at a higher-dimensional version of \(b(x) = x\log(x)^{1+c}\).
Lastly, Section \ref{sec:numExp_ReacDiff} considers estimating the blow-up time of a semi-discretized reaction-diffusion equation. 
The experiments presented in this section were performed in Julia~\cite{Bezanson2017} and the plots were generated with the Makie library~\cite{Danisch2021}.

\subsection{Note on methods used for comparison}
We will in all the bellow examples compare our adaptive method to using uniform steps and the method introduced by Hirota and Ozawa in \cite{Hirota2006}. 
This method is based on rewriting the underlying ODE into a system for the arc length of the process and then solving this system using a higher-order ODE solver. 
There results assume that the blow-up component satisfies \(\norm{x(t)} = \Theta\left(\frac{1}{(\tau-t)^p}\right)\) for some \(p > 0\). 

In the following figures we plot the observed error for this method against a tolerance \(\epsilon\) that is used as an input for the other numerical methods. 
As the method presented by Hirota and Ozawa does not take a tolerance input we use this value to decide how far we are calculating the arc length. 
The observed error and cost for this method must therefore be read together in order to see how quickly the error is decreasing relative to the increase in cost. 

Due to the use of a fifth-order ODE-solver this method performs impressively, achieving quick convergence for many of the problems we study. However, our adaptive method has some theoretical advantages, as we under less restrictive regularity assumptions on the vector field \(b\) and without any a priori assumptions on the form of the blow-up components have, in contrast to \cite{Hirota2006}, provided rigorous error and complexity bounds for the numerical method. 
Following Remark \ref{rem:ho_sceme}, one may also in some situations use our method with higher-order integrators in order to match the performance of the method by Hirota and Ozawa.
Additionally, our method is not dependent on any parameters that need to be tweaked.

\subsection{Comparing several methods for a 1D ODE}
\label{sec:numExp_UniformVsAdaptive}
In this section we want to verify that Algorithm \ref{alg:1d} provides better results than relying on uniform time steps. 
From the discussion in Section \ref{sec:adaptiveVsUniform1D}, we would expect the cost to be higher in the uniform case. 
Furthermore, in addition to the method by Hirota and Ozawa we also consider 
two alternative adaptive approaches with the first being by Stuart and Floater \cite{Stuart90}. It utilizes a rescaling of the time variable to estimate the explosion time. The other is the modified rescaling algorithm presented by Cho and Wu \cite{Cho2024}, which repeatedly rescales the numerical solution when it reaches a predefined threshold value \(M\), yielding a new problem that is solved at a finer resolution near the blow-up time. 
And to illustrate the potential of higher-order, sensitivity-based a priori adaptivity, 
we also include a second-order method of the form detailed in Remark \ref{rem:ho_sceme}.

For this example we consider \(b(x) = x^2\), giving \(\tau(x_0) = \frac{1}{x_0}\). 
We fix \(x_0 = \frac{1}{2}\) and \(k = 1.1\), implying that \(\tau = 2\). 
Note that Assumption \ref{asmp:b} is satisfied if we choose \(F(b(x)) = G(b(x)) = \frac{1}{x}\), as we then have \(-\frac{d}{dx}F(b(x)) = -\frac{d}{dx}\frac{1}{x} = \frac{1}{x^2} = \frac{1}{b(x)}\) as required. 
\begin{figure}
\center
  \includegraphics[width=\linewidth]{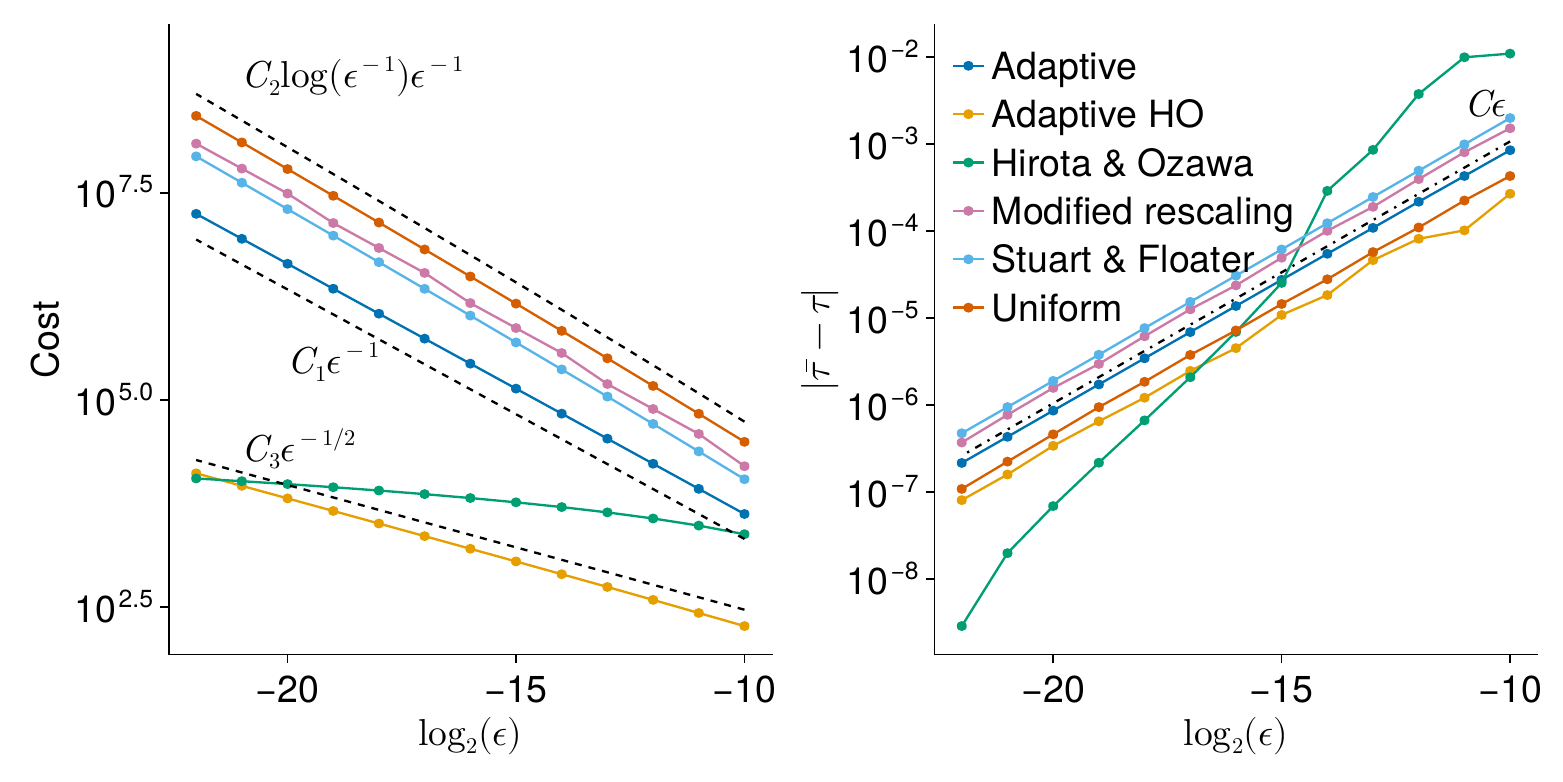}
  \caption{Estimation of blow-up time for \(x'(t) = x(t)^2\). Left: Computational cost plotted against the tolerance. Right: Absolute error plotted against the tolerance.}
  \label{fig:adaptive_vs_uniform}
\end{figure}

The results of our experiment is shown in Figure \ref{fig:adaptive_vs_uniform}. 
From the left plot we see that the rates for the cost are as expected, with the adaptive method outperforming the uniform time stepping with a log-factor, and its higher-order version showing an even lower cost of order \(\mathcal{O}\left(\epsilon^{-\frac{1}{2}}\right)\).
We also note that the method presented by Stuart and Floater in \cite{Stuart90} perform similarly to the uniform approach, which is expected as it can easily be derived from the proof of \cite[Result 3.4]{Stuart90}. 
The same order is observed for the modified rescaling algorithm given by Cho and Wu \cite{Cho2024}, which also is as expected when observing that the number of rescaling cycles scales like \(\mathcal{O}(\log(\epsilon^{-1}))\). 
Moreover, the higher-order method given by Hirota and Ozawa \cite{Hirota2006} yields better results than the other methods. 
From the right part of the figure we see that the error of our estimate is of the expected order for all methods, with Hirota and Ozawa outperforming the alternative methods due to the use of a fifth order method. We see that our theoretical results from Theorem \ref{thm:error_alg1} are reflected, and that adaptive time stepping is advantageous to uniform time stepping. 

\subsection{Using the adaptive method on a 1D ODE without analytical solution}
\label{sec:numExp_1DNoAnaSolution}
Here we test our adaptive method on a problem for which we do not have an analytical solution, and let \(b(x) = e^{x^2}\). 
Note that calculating \(\frac{\tau}{\sqrt{\pi}}\) is equivalent to estimating \(P(X \geq x_0)\) for a random variable \(X \sim N(0, \frac{1}{2})\). 
Additionally, for this choice of \(b\), Assumption \ref{asmp:b} is met with the alternative to Assumption \ref{asmp:b.FandG} given in Remark \ref{rem:followAltAsmpThm1}. 
We fix \(x_0 = 1\) and \(k = 1.1\), and calculate a pseudo solution using our adaptive method over a grid generated using the tolerance \(\epsilon=2^{-33}\).

\begin{figure}
\center
  \includegraphics[width=\linewidth]{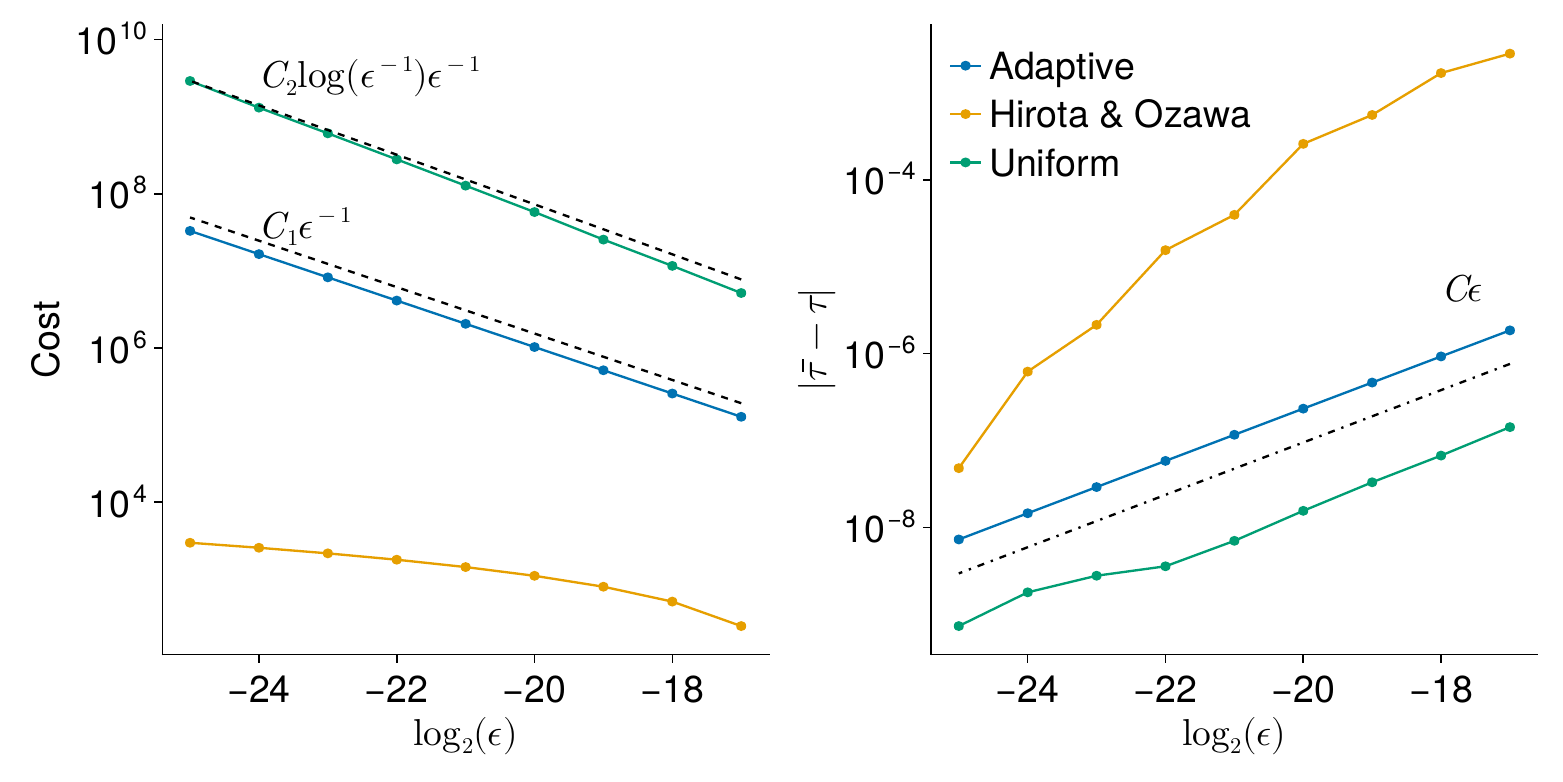}
  \caption{Estimation of blow-up time for \(x'(t) = e^{x(t)^2}\). Left: Computational cost plotted against the tolerance. Right: Absolute error with regards to a pseudo solution plotted against the tolerance.}
  \label{fig:exp_x_x}
\end{figure}
The results are shown in Figure \ref{fig:exp_x_x}, and similarly to what we saw in the previous numerical experiment, the left plot shows that we get the cost that we expect for both methods In the right plot of Figure \ref{fig:exp_x_x} we see that the error has the order we expected over this domain for the adaptive and uniform method. The higher-order method of Hirota and Ozawa again outperform the others by achieving a steeper decline in error with a lower increase in cost. 

\subsection{Extending the adaptive method for 1D ODEs}
\label{sec:numExp_xlogx_c}
In this example we consider \(b(x) = x\log(x)^{1+c}\) for \(c>0 \). Then~\ref{asmp:b.integrabilityCond} is 
violated and this can affect the performance asymptotics of our adaptive method, as 
the upper bound for~\eqref{eq:costEst1D} in the cost estimate for Algorithm~\ref{alg:1d},
now becomes substantially larger. To control the approximation error, we define \(r(\epsilon) = e^{(c\epsilon)^{-\frac{1}{c}}}\), 
as this yields  
\[
|\tau - \tau_\epsilon| = \int_{r(\epsilon)}^\infty\frac{1}{b(x)}\,dx = \epsilon.
\]
Moreover, from \eqref{eq:secondTermExp} we note that it suffices to set \(h_{n-1} = \sqrt{\frac{\epsilon}{Nb'(x_{n-1})}}\) to achieve \(\abs{\bar\tau-\tau_r} = \mathcal{O}(\epsilon)\). 
Then, disregarding higher-order terms, we get \(\sqrt{N} = \frac{1}{\sqrt{\epsilon}}\int_{x_0}^{r(\epsilon)}\frac{\sqrt{b'(kx)}}{b(x)}\,dx\), which implies that 
\[N = \begin{cases}
    \mathcal{O}\left(\epsilon^{-\frac{1}{c}}\right) & c < 1\\
    \mathcal{O}\left(\epsilon^{-1}\log(\epsilon^{-1})^2\right) & c = 1
\end{cases},\] 
as when $c \in (0,1)$, we have
\begin{align*}
    \int_{x_0}^{r(\epsilon)}\frac{\sqrt{b'(kx)}}{b(x)}\,dx &= \int_{x_0}^{r(\epsilon)}\frac{\sqrt{\log(kx)^{1+c} + (1+c)\log(kx)^c}}{x\log(x)^{1+c}}\,dx \nonumber \\
    &\leq \int_{x_0}^{r(\epsilon)}\frac{C'\log(x)^{\frac{1+c}{2}}}{x\log(x)^{1+c}}\,dx \nonumber \\ 
    %&\leq \int_{x_0}^{r(\epsilon)}\frac{C'}{x\log(x)^{\frac{1+c}{2}}}\,dx\\
    & \leq C'\log(r(\epsilon))^{\frac{1-c}{2}},
\end{align*}
for some constant \(C'>0\). When \(c = 1\), similar reasoning yields the upper bound \(C'\log(\log(r(\epsilon)))\). 

For the uniform time stepping, we recall from Section~\ref{sec:adaptiveVsUniform1D} that setting \(h=\frac{\epsilon}{\log(b(r))} = \mathcal{O}\left(\epsilon^{1 + \frac{1}{c}}\right)\) leads to our desired error rate of \(\mathcal{O}(\epsilon)\). However, now we get a significant higher cost using uniform steps, with \(N = \mathcal{O}(\epsilon^{-1-\frac{1}{c}})\). 

\begin{figure}
\center
  \includegraphics[width=\linewidth]{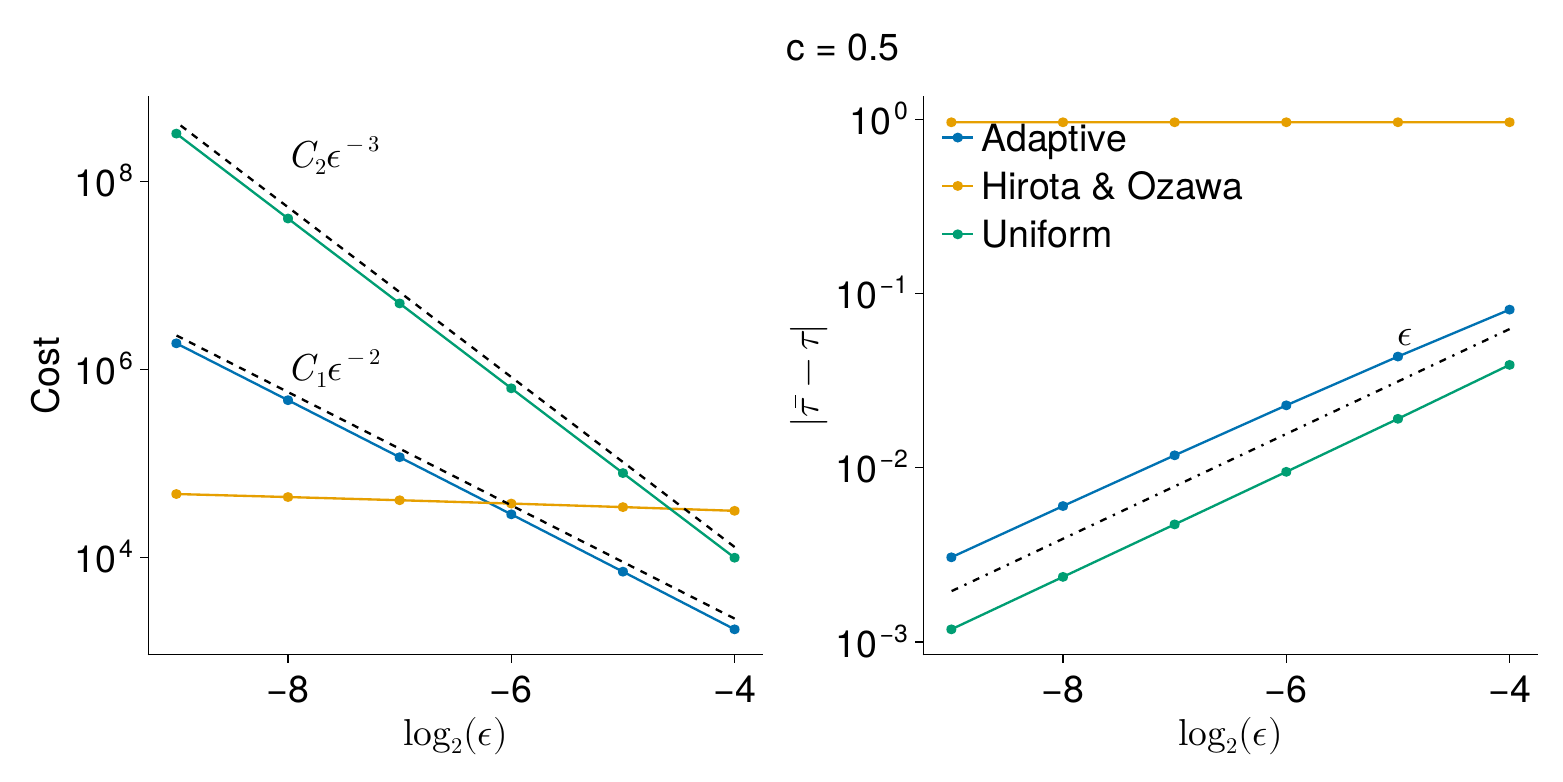}
  \includegraphics[width=\linewidth]{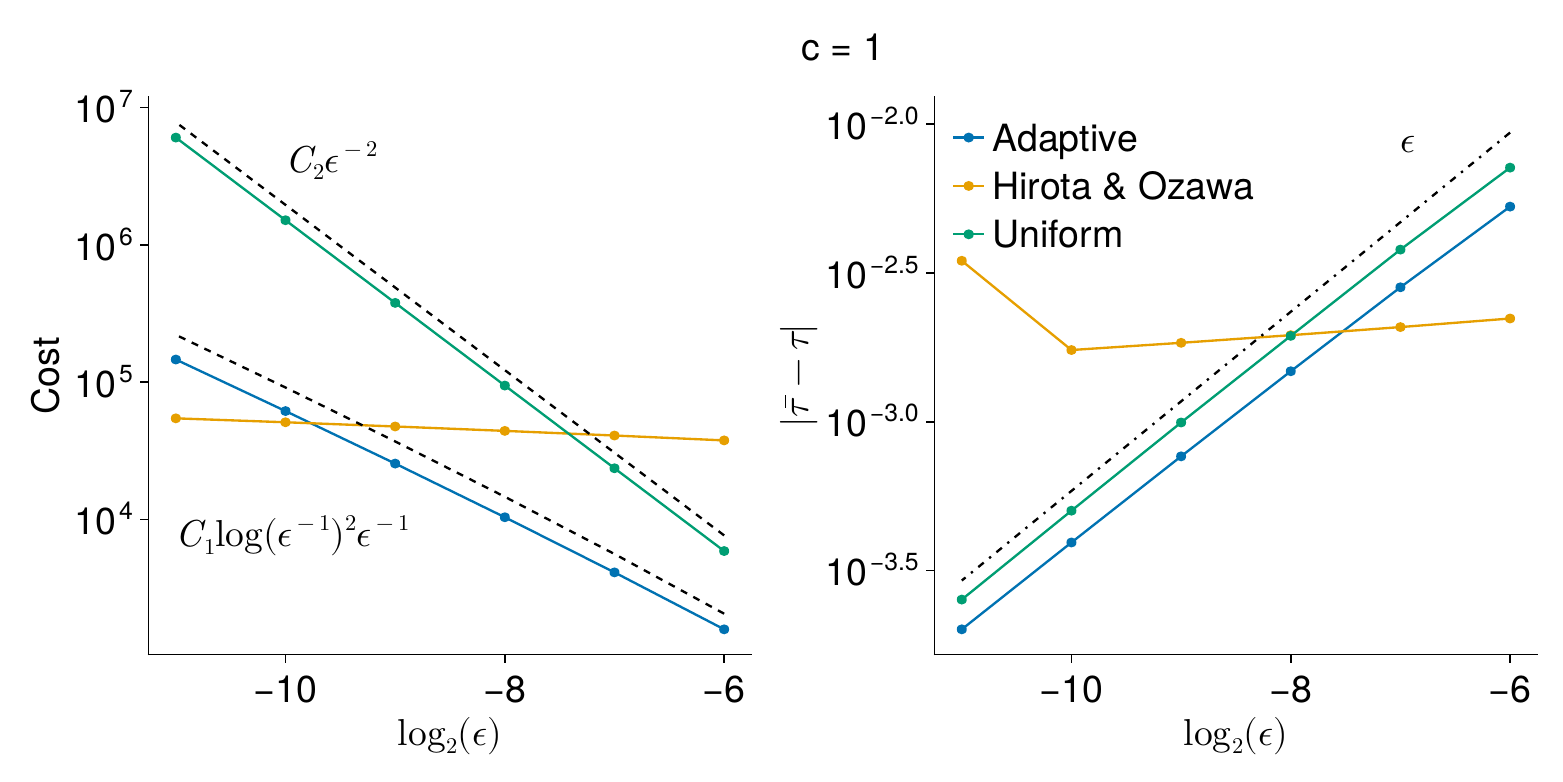}
  \caption{Estimation of blow-up time for \(x'(t) = x(t)\log(x(t))^{1+c}\). Left plots: Computational cost plotted against the tolerance. Right plots: Absolute error plotted against the tolerance.}
  \label{fig:xlogx_1.5.pdf}
\end{figure}

We run this experiment for both \(c=\frac{1}{2}\) and \(c = 1\).
We fix \(x_0 = 2\), which gives \(\tau = \frac{1}{c\log(2)^c}\).
The results are shown in Figure \ref{fig:xlogx_1.5.pdf}. 
From the left plots we see that the cost of using uniform and adaptive steps are as expected for both values of \(c\). 
The same is true for the rate of the error which we from the right plots see that have the expected rate of \(\mathcal{O}(\epsilon)\). 
In this example we see that the error rate for the higher-order method does not resemble what we observed in earlier examples. This is probably due to the fact that \(x(t)\) does not have the required form \(x(t) = \Theta\left(\frac{1}{(\tau-t)^p}\right)\). 

\subsection{An uncoupled problem in multiple dimensions}
\label{sec:numExp_UncoupledProblem}
We consider a simple system of ODEs where we are able to analytically find the blow-up time of \(\abs{x}\) in order to numerically test Algorithm \ref{alg:Rn}. We consider
\[b(x) = 
    \begin{bmatrix}
        x_1^3 \\
        x_2^5 
    \end{bmatrix},
\]
with \(x(0) = (\sqrt{2}, 1)\). 
We let \(D=\R^2\setminus B(\sqrt{3})\).
An illustration is given in Figure \ref{fig:ex_simple_multidim}.

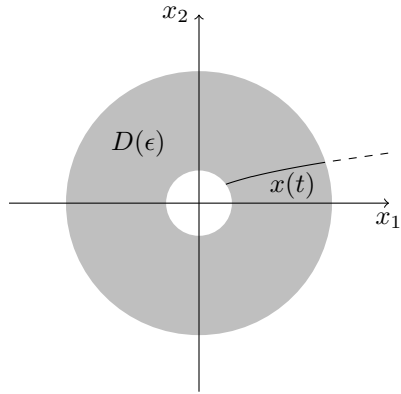
\begin{figure}
    \centering
    \begin{tikzpicture}[scale=0.25]
        \fill[fill=gray!50!white] (0,0) circle [radius=7cm];
        \fill[fill=white] (0,0) circle [radius=1.732cm];
        %\draw[dashed] (1.4142, 1) to[bend right] node[below right] {\(x(t)\)} (45: 7);
        \draw[domain=0:0.239, samples=100] plot ({(1/2-2*\x)^(-1/2)}, {(1-4*\x)^(-1/4)}) node[below left] {\(x(t)\)};
        \draw[domain=0.24:0.245, dashed, samples=100] plot ({(1/2-2*\x)^(-1/2)}, {(1-4*\x)^(-1/4)});
        %\draw[domain=0:0.062553, dashed, samples=50] plot ({(1-2*\x)^(-1/2)}, {(1/4-4*\x)^(-1/4)}) node[below left] {\(x(t)\)};
        \node at (135: 4.5) {\(D(\epsilon)\)};
        \draw[->] (-10,0) -- (10,0) node[below] {\(x_1\)};
        \draw[->] (0,-10) -- (0,10) node[left] {\(x_2\)};
    \end{tikzpicture}
    \caption{Illustration of example studied in Section \ref{sec:numExp_UncoupledProblem}}
    \label{fig:ex_simple_multidim}
\end{figure}

Here \(x_1\) and \(x_2\) are uncoupled, allowing us to find explicit expressions for \(x_1(t)\) and \(x_2(t)\). 
We get \(x_1(t) = \text{sign}(x_1(0))\left(\frac{1}{x_1(0)^2} -2t\right)^{-\frac{1}{2}}\) and \(x_2(t) = \text{sign}(x_2(0))\left(\frac{1}{x_2(0)^4} -4t\right)^{-\frac{1}{4}}\). 
Thus we have \(\tau(x) = \min\left\{\frac{1}{2x_1^2}, \frac{1}{4x_2^4}\right\}\), giving \(\tau=\frac{1}{4}\). 
Furthermore, 
\[b'(x) = 
    \begin{bmatrix}
        3x_1^2 & 0 \\
           0 & 5x_2^4 
    \end{bmatrix}.
\]
From \(b\) and \(b'\) we see that assumptions \ref{asmp:Rn.lwrAndUprBndGrowthSqrdNorm}, \ref{asmp:Rn.directionalityOfB.sign}, \ref{asmp:Rn.bDerOverBBound} and \ref{asmp:Rn.bndNormBDirection} are satisfied, with parameters \(\alpha = 2, \nu = 2\) and \(\upsilon = 1\).

To show that Assumption \ref{asmp:Rn.BndGradTau} holds, we use \eqref{eq:gradTauDir} to calculate \(\nabla\tau_\epsilon(x)\), giving
\[\nabla\tau_\epsilon(x) = -\left(\frac{\text{sign}(x_1(0))}{x_1(0)^3}\frac{x_1(\tau)^4}{x_1(\tau)^4 + x_2(\tau)^6}, \frac{\text{sign}(x_2(0))}{x_2(0)^5}\frac{x_2(\tau)^6}{x_1(\tau)^4 + x_2(\tau)^6}\right),\]
so there exist \(C>0\) such that \(\abs{\nabla\tau(x)} \leq \frac{C}{\abs{b(x)}}\).

Lastly, in order to see that Assumption \ref{asmp:Rn.bndBdNextByPrev} holds we consider the term \(\frac{\epsilon}{\sqrt{\norm{b'(x)}}}b(x)\). From Lemma \ref{lemma:boundTauDiff} we have \(r(\epsilon) = \mathcal{O}(\epsilon^{-\frac{1}{\alpha}}) = \mathcal{O}(\epsilon^{-\frac{1}{2}})\), implying that \(\epsilon = r(\epsilon)^{-2} \leq \abs{x(t)}^{-2}\). 
Thus
\[\frac{\epsilon}{\sqrt{\norm{b'(x)}}}b(x) \leq \frac{1}{\left(\max\left\{2x_1^2, 5x_2^4\right\}\right)^{\frac{1}{2}}(x_1^2+x_2^2)}[x_1^3, x_2^5].\]
Using that \(\norm{x} \geq \sqrt{3}\) in our case, one can show that the expression above can be bounded by \(c[x_1, x_2]\) for some constant \(c\geq1\). 
Then, as \(\norm{b'(x)}\) is increasing in both components of \(x\), our inequality becomes \(\norm{b'(x + \theta b(x)h(x))} \leq \norm{b'(cx)} \leq k\norm{b'(x)}\), which for any \(\theta\in(0,1)\) holds for some constant \(k\geq1\). 

For the uniform time steps we from \eqref{eq:boundHD} get that
\[\abs{\bar\tau(x_0)-\tau_{\epsilon}(x_0)} \leq kC\sum_{n=1}^N \frac{\norm{b'(\bar x_{n-1})}\abs{b(\bar x_{n-1})}}{\abs{b(\bar x_{n-1})}}h^2 \leq Ch\sum_{n=1}^N \abs{\bar x_{n-1}}^{\gamma -\alpha-1}\Delta\abs{\bar x}_{n-1}.\] 
Now, as we in this example have \(\gamma-\alpha-1  \geq 0\), we get
\[\abs{\bar\tau(x_0)-\tau_{\epsilon}(x_0)} \leq Ch\int_{\abs{x_0}}^{r(\epsilon)}\abs{\bar x}^{\gamma - \alpha - 1}\,d\abs{x} \leq Chr(\epsilon)^{\gamma-\alpha} \leq Ch\epsilon^{1-\frac{\gamma}{\alpha}}.\]
Equating to \(\epsilon\) and solving for \(h\) while ignoring constants gives \(h = \epsilon^\frac{\gamma}{\alpha} = \epsilon^{2}\). This yields a computational cost of order \(\mathcal{O}\left(\epsilon^{-\frac{\gamma}{\alpha}}\right) = \mathcal{O}\left(\epsilon^{-2}\right)\). 

However, note that which component that explodes depends on \(x(0)\), and if \(x_2(t)\) blows up first we in practice have \(\alpha = 4\). Then \(\gamma-\alpha-1 = -1\), which gives an alternative expression for \(h\). Later, in Section \ref{sec:numExp_CoupledProblem}, we show that \(h = \frac{\epsilon}{\log(r(\epsilon))}\) is sufficient in this situation in order to have \(\abs{\bar\tau-\tau}=\mathcal{O}(\epsilon)\), with a cost of \(\mathcal{O}(\epsilon^{-1})\log(\epsilon^{-1})\). For our choice of \(x(0)\) both components explode at \(t = \frac{1}{4}\).
We test both choices of uniform \(h\) and call the latter of these Log Uniform.  

\begin{figure}
\center
  \includegraphics[width=\linewidth]{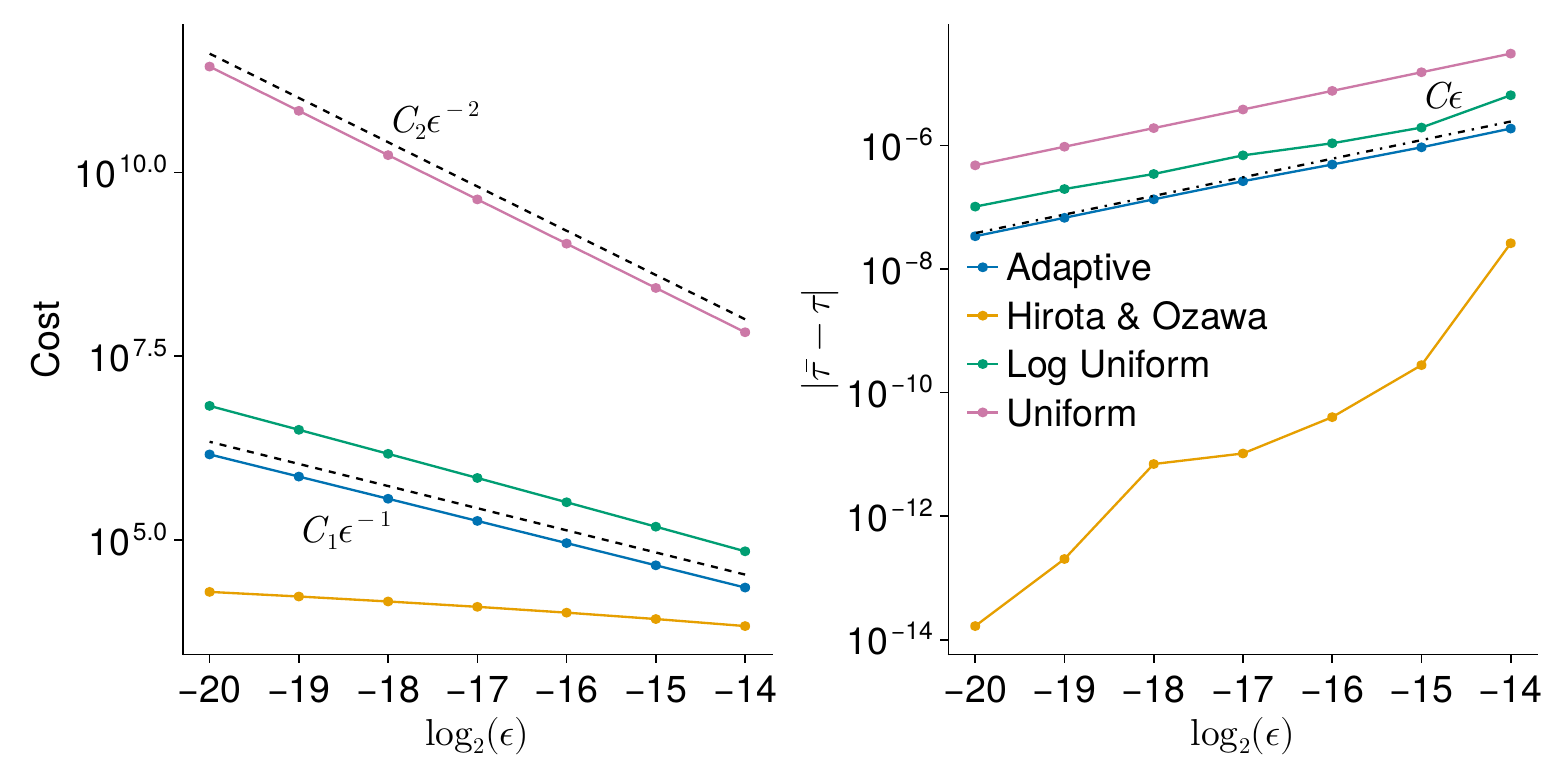}
  \caption{Estimation of blow-up time for the uncoupled multidimensional problem studied in Section \ref{sec:numExp_UncoupledProblem}. Left: Computational cost plotted against the tolerance. Right: Absolute error plotted against the tolerance.}
  \label{fig:simple_multidim}
\end{figure}

The cost of the different methods are show in the left part of Figure \ref{fig:simple_multidim}, and are as expected, with the log uniform version having a cost of a order slightly higher than linear. The error rate is also consistent with the theory. Again, the higher-order method of Hirota and Ozawa outperforms the lower-order methods. 

\subsection{Estimating blow-up of a coupled multidimensional problem}
\label{sec:numExp_CoupledProblem}
In this experiment we look at a system of ODEs where the two components are coupled. We let 
\[b(x) = \begin{bmatrix}
        x_1^3 + x_1x_2^2 \\
        x_2^3 + x_1^2x_2
         \end{bmatrix},\]
and fix both \(x(0) = (1, 2)\) and \(D = \R^2\setminus B(\sqrt{5})\). The Jacobi matrix \(b'(x)\) is then
\[b'(x) = \begin{bmatrix}
        3x_1^2 + x_2^2 & 2x_1x_2 \\
        2x_1x_2 & x_1^2 + 3x_2^2
         \end{bmatrix}.\]
As \(b'(x)\) is symmetric we have that \(\norm{b'(x)}_2 = \lambda_{\max}(b'(x)) = 3(x_1^2+x_2^2)\). Simple calculations show that assumptions \ref{asmp:Rn.lwrAndUprBndGrowthSqrdNorm}, \ref{asmp:Rn.directionalityOfB.ineq} and \ref{asmp:Rn.bndNormBDirection} are met with \(\alpha = 2\). In addition, Lemma \ref{lemma:suffCondAsmpbDerOverBBound} holds with \(\gamma=2\), implying that Assumption \ref{asmp:Rn.bDerOverBBound} holds with \(\nu=2\) and \(\upsilon=1\). We use the sufficient condition described in Remark \ref{rem:suffCondAsmpFirstVar} to show that Lemma \ref{lemma:SuffCondAsmpGradTau} and therefore Assumption \ref{asmp:Rn.BndGradTau} holds. We have \(\underset{v\in \bar B\left(\abs{b(x)}\right)}{\sup} v^Tb'(x)v = \frac{1}{2}\lambda_{\max}\left(b'(x)+b'(x)^T\right)\abs{b(x)}^2\). Thus \eqref{eq:suffCondAssmpFirstVar} gives
\[3(x_1^2+x_2^2)^4 \leq C^2b(x)^Tb'(x)b(x) = 3(x_1^2+x_2^2)^4.\] Moreover, note that \(C\abs{b(x_0)} = \sqrt{5}\geq1\) and \(b(x_0)b'(x_0)b(x_0)-\lambda_{\max}(b'(x_0)) > 0\). Lastly, for Assumption \ref{asmp:Rn.bndBdNextByPrev} we first note that from Lemma \ref{lemma:boundTauDiff}\ we have \(\epsilon =\mathcal{O}(r^{-2}) \geq \mathcal{O}((x_1^2 + x_2^2)^{-1})\). Then \(\epsilon\frac{1}{\sqrt{\norm{b'(x)}}}b(x) = \frac{1}{\sqrt{3(x_1^2 + x_2^2)}}[x_1, x_2]\). 
This gives
\[\norm{b'\left(x +\epsilon\frac{\theta}{\sqrt{\norm{b'(x)}}}b(x) \right)}_2 \leq 3\left(\sqrt{x_1^2 + x_2^2}+\frac{1}{\sqrt{3}}\right)^2 \leq 3k(x_1^2 + x_2^2) = k\norm{b'(x)}_2.\]

For the uniform time steps, note that \(\alpha = \gamma\) gives \(\gamma-\alpha-1 = -1\). Assume that we have constants \(\hat{C}, \beta>0\) such that \(b(x)\cdot x \leq \hat{C}\abs{x}^{2+\beta}\). 
Then, starting from \eqref{eq:boundHD} we get
\begin{align*}
    \abs{\bar\tau(x_0)-\tau_{\epsilon}(x_0)} &\leq Ch\sum_{n=1}^N \abs{\bar x_{n-1}}^{\gamma -\alpha-1}\Delta\abs{\bar x}_{n-1} \\
    &= Ch\sum_{n=1}^N \left(\abs{\bar x_{n}}^{\gamma -\alpha-1} + \abs{\bar x_{n-1}}^{\gamma -\alpha-2}\Delta\abs{\bar x}_{n-1}\right)\Delta\abs{\bar x}_{n-1} \\
    &\leq Ch\int_{\abs{x_0}}^{r(\epsilon)}\abs{\bar x_n}^{\gamma-\alpha-1}\,d\abs{x} + Ch^2\sum_{n=1}^N\abs{\bar x_{n-1}}^{\gamma + \beta - \alpha - 1}\Delta\abs{\bar x}_{n-1} \\
    &\leq Ch\log(r(\epsilon)) + Ch^2r(\epsilon)^\beta,
\end{align*}
where the jump to the second line adds and subtracts \(\abs{\bar x_{n}}^{\gamma -\alpha-1}\) before taking a Taylor expansion of the subtracted term.
Setting each term equal to \(\epsilon\) and solving for \(h\) gives that setting \(h = \min\left\{\frac{\epsilon}{\log\left(\epsilon^{-\frac{1}{\alpha}}\right)}, \epsilon^{\frac{1}{2}(\frac{\beta}{\alpha}+1)}\right\} = \frac{\epsilon}{\log\left(\epsilon^{-1}\right)}\) ensures \(\abs{\bar\tau(x_0)-\tau_{\epsilon}(x_0)} = \mathcal{O}(\epsilon)\). 
The cost is then \(Cost = \mathcal{O}\left(\epsilon^{-1}\log\left(\epsilon^{-1}\right)\right)\). 
From Theorem \ref{thm:propertiesAlgRn} we expect the same error rate for the adaptive method, but at a lower cost of \(\mathcal{O}(\epsilon^{-1})\). 

From the left plot of Figure \ref{fig:coupled_problem} we see that the costs are as expected, with again the higher-order method outperforming the first-order methods for smaller \(\epsilon\). The same is observed for the error rates in the right plot of the figure. Here the error is calculated against a pseudo solution generated using the adaptive method with \(\epsilon=2^{-27}\).
\begin{figure}
\center
  \includegraphics[width=\linewidth]{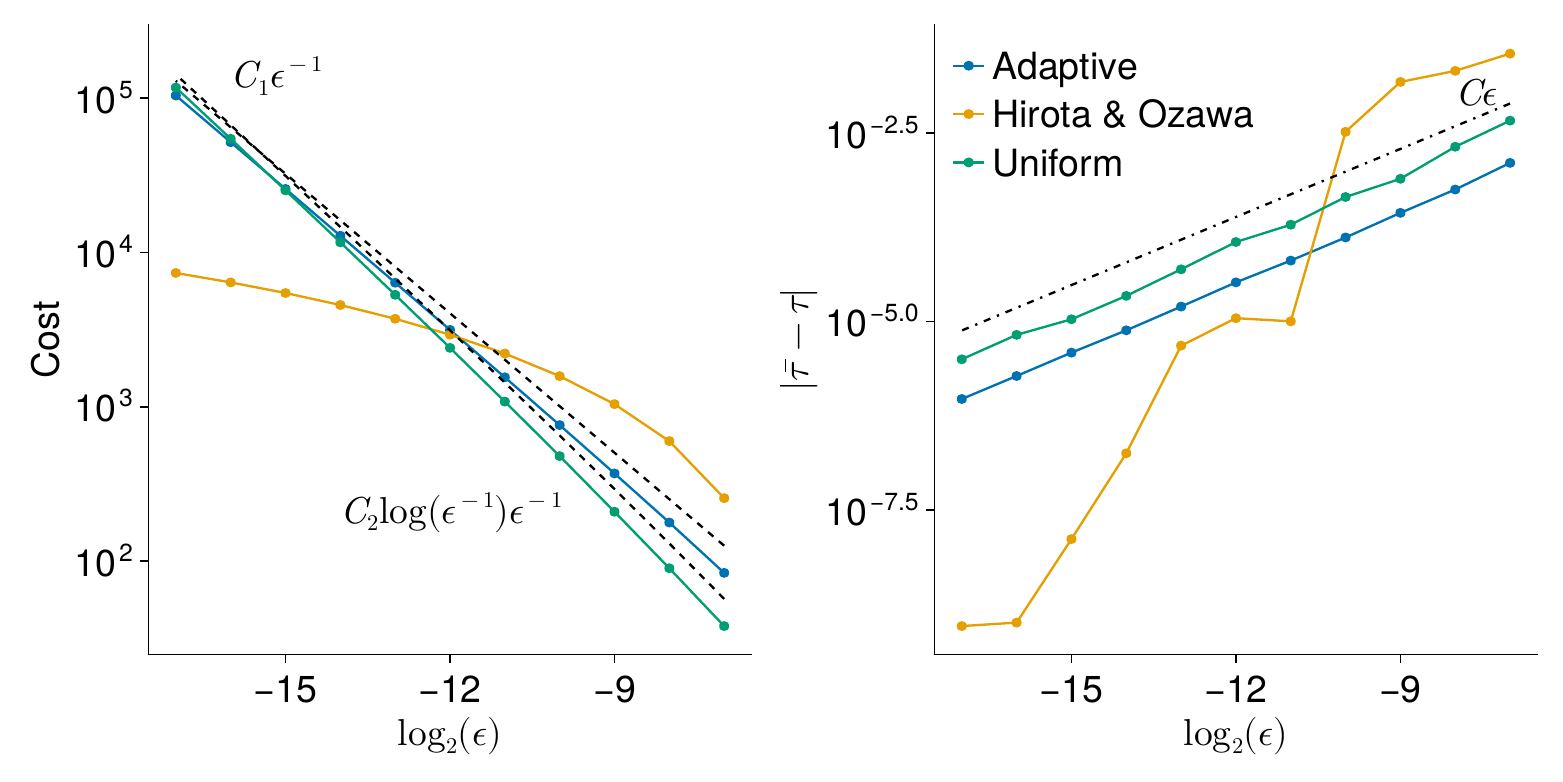}
  \caption{Estimation of blow-up time for the coupled multidimensional problem studied in Section \ref{sec:numExp_CoupledProblem}. Left: Computational cost plotted against the tolerance. Right: Absolute error with regards to a pseudo solution plotted against the tolerance.}
  \label{fig:coupled_problem}
\end{figure}

\subsection{Estimating blow-up of slow processes in multiple dimensions}
\label{sec:numExp_slowGrowthMultidim}
Similarly to Section \ref{sec:numExp_xlogx_c}, we here want to considerer functions \(b:\R^n\to\R^n\) that explodes relatively slow. Let 
\[b(x) = 
    \begin{bmatrix}
        x_1\log(x_1^2 + 2x_2^2)^{1+c} \\
        x_2\log(2x_1^2 + x_2^2)^{1+c} 
    \end{bmatrix},\]
and set \(x(0) = (4, 3)\). We let \(D = \R^2\setminus B(5)\). 
This function satisfies Assumption \ref{asmp:RnAlt} with \(\alpha  = c, \gamma = 1 + c\) and some finite values for the constants \(\check{C}\) and \(C_3\). 
Assumption \ref{asmp:Rn.bndNormBDirection} also trivially holds.  
We let \(c \in \left\{\frac{1}{2}, 1\right\}\). 
For Assumption \ref{asmp:Rn.directionalityOfB.ineq} it suffices to note that  
\[\frac{b(x)\cdot x}{\abs{b(x)}\abs{x}} \geq \frac{\log(x_1^2 + x_2^2)^{1+c}}{\log(2(x_1^2 + x_2^2))^{1+c}} > 0.\]
Let \(R \coloneqq x_1^2 + x_2^2\). Then, for Assumption \ref{asmp:Rn.BndGradTau} we have that \eqref{eq:suffCondAssmpFirstVar} holds as
\begin{equation*}
    \lambda_{\max}(b'(x))\abs{b(x)}_2^2 \leq CR^2\log(2R^2)^{3+3c} \leq C'^2 b(x)^Tb'(x)b(x),
\end{equation*}
for some constants \(C, C'>0\).
Therefore Assumption \ref{asmp:Rn.BndGradTau} is satisfied as the above together with Assumption \ref{asmp:Rn.directionalityOfB.ineq} is sufficient for Lemma \ref{lemma:SuffCondAsmpGradTau} to hold.
For Assumption \ref{asmp:Rn.bndBdNextByPrev} we get \(\frac{\epsilon}{\sqrt{\norm{b'(x)}_2}}b(x) \leq \frac{c\epsilon b(x)}{\log(R^2)^{\frac{1+c}{2}}}.\)
One may then show that
\[\frac{\norm{b'\left(x + \frac{c\theta\epsilon b(x)}{\log(R^2)^{\frac{1+c}{2}}}\right)}_2}{\norm{b'(x)}_2} \leq k,\]
for some constant \(k\), which is what we need. 

For the uniform steps, we consider the error indicator we found in the proof of Theorem \ref{thm:propertiesAlgRn}, giving
\[\abs{\bar\tau - \tau_\epsilon} \leq kC\sum_{n=1}^{N}\frac{\norm{b'(\bar x_{n-1})}\abs{b(\bar x_{n-1})}}{\abs{b(\bar x_{n-1})}}h^2 \leq \frac{kCh}{C_2}\sum_{n=1}^{N}\frac{\norm{b'(\bar x_{n-1})}}{\abs{b(\bar x_{n-1})}}\Delta\abs{\bar x} \leq \frac{kCh}{C_2}\log(r)\]
as \(\gamma - \alpha - 1 = 0\). Then, in order to bound the error by \(\epsilon\), we must use the uniform step \(h = C\frac{\epsilon}{\log(r(\epsilon))}\). From Corollary \ref{corr:boundTauDiffAlt} we have that \(r(\epsilon) = e^{\left(\frac{1}{\check{CC_2c\epsilon}}\right)^{\frac{1}{c}}}\), thus the cost of using uniform steps is \(N = \mathcal{O}\left(\epsilon^{-3}\right)\) for \(c = \frac{1}{2}\) and \(N = \mathcal{O}(\epsilon^{-2})\) when \(c = 1\).

From Corollary \ref{corr:propertiesAlgRnLog} we have that the same error is achieved using the adaptive method, but at the lower cost of \(N = \mathcal{O}(\epsilon^{-2})\) when \(c = \frac{1}{2}\) and \(N = \mathcal{O}(\epsilon^{-1}\log(\epsilon^{-1})^2)\) for \(c = 1\). 

The results from these experiments are shown in Figure \ref{fig:log_multidim.pdf}. The pseudo solution used to calculate the error is generated using the adaptive method with \(\epsilon=2^{-17}\). From the figure we see that both methods have the expected error of order \(\mathcal{O}(\epsilon)\) for both values of $c$, and from the plots on the left we see that the cost is as expected. Similarly to what we observed in Section \ref{sec:numExp_xlogx_c}, the higher-order method does not perform well here, again most likely due to the blow-up component not having the required form. 

\begin{figure}
\center
  \includegraphics[width=\linewidth]{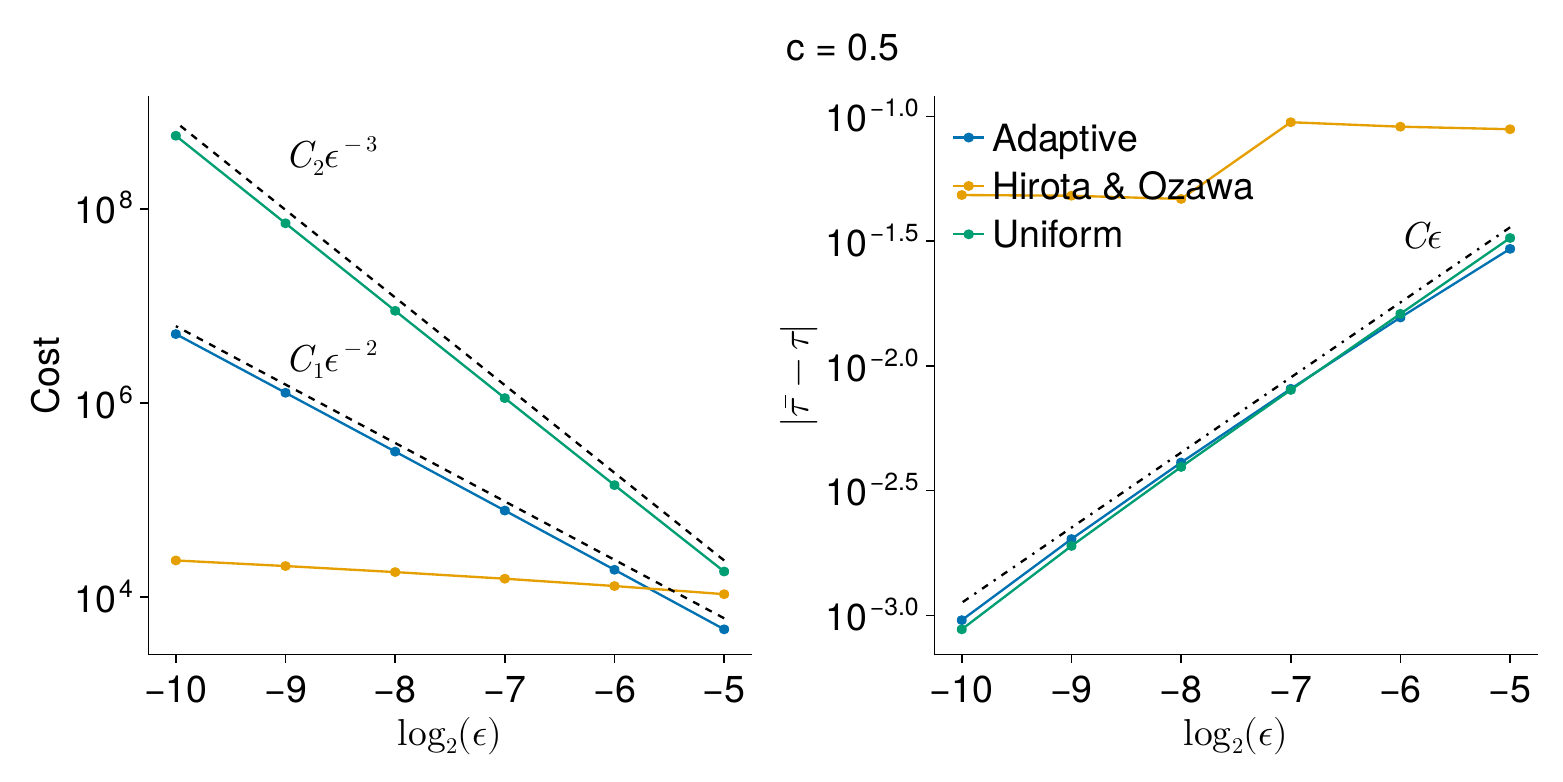}
  \includegraphics[width=\linewidth]{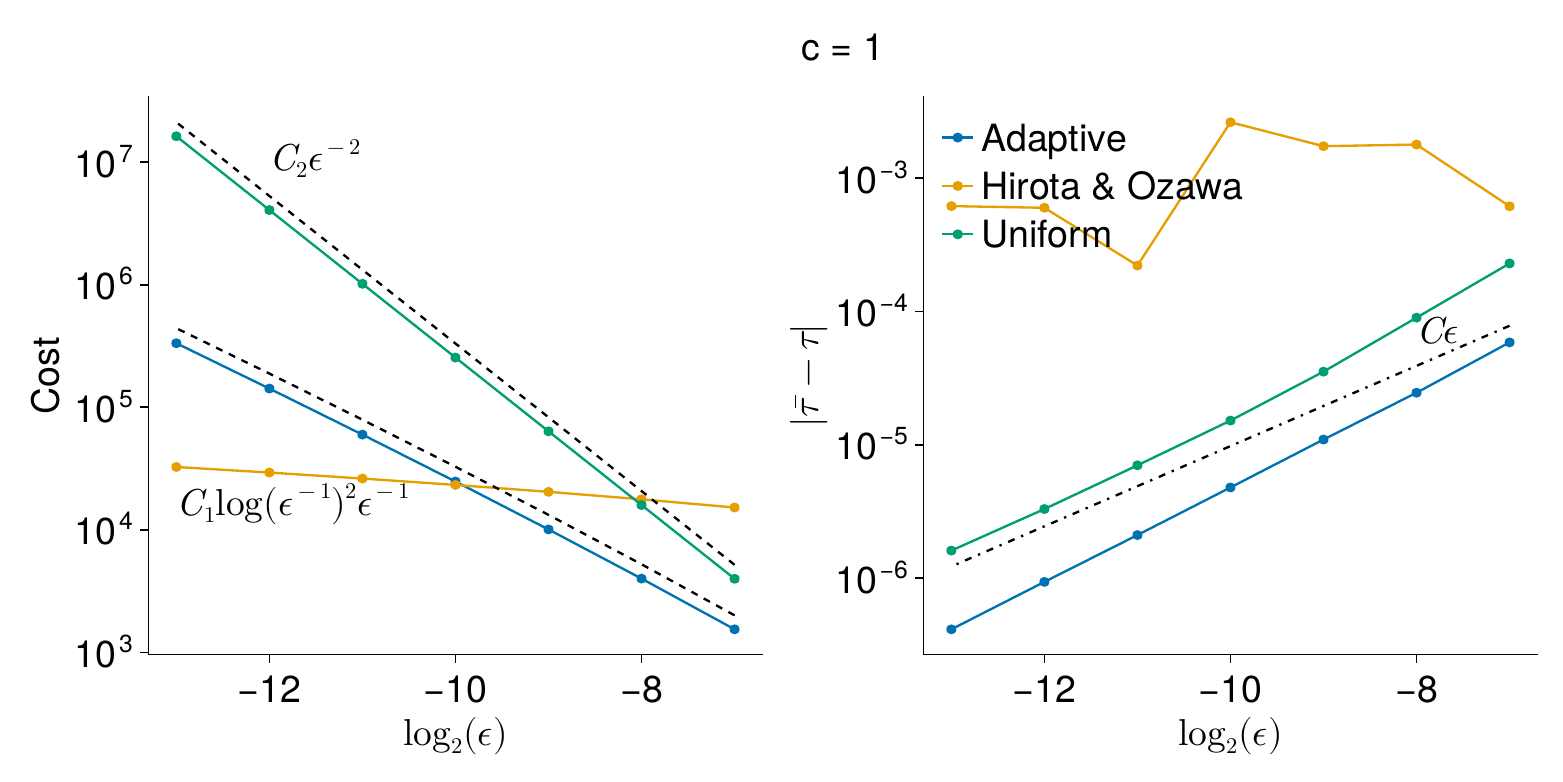}
  \caption{Estimation of blow-up time for the problem studied in Section \ref{sec:numExp_slowGrowthMultidim}.
  Left: Computational cost plotted against the tolerance. Right: Absolute error with regards to a pseudo solution plotted against the tolerance.}
  \label{fig:log_multidim.pdf}
\end{figure}

\subsection{Semi-discretized reaction-diffusion equation}
\label{sec:numExp_ReacDiff}

In the last example, we estimate the blow-up time of a semi-discretized version of a reaction-diffusion equation with a quadratic nonlinearity. We are not able to verify that this problem satisfies Assumption \ref{asmp:Rn}, so this is a purely experimental numerical study.
The PDE is given by \(\frac{\partial u}{\partial t}(t, x) = \Delta u(t, x) + f(u(t, x))\) for \((t,x) \in (0,\infty) \times (0,1)\) with \(f(x) = x^2\), Dirichlet-0 boundary conditions \(u(t, 0) = u(t, 1) = 0\) and the initial condition \(u(x, 0) = u_0(x) = 100\sin(\pi x)\).

The semi-discretized version of this PDE, with finite-difference approximation in space, is the following system of ODEs in $\R^{m-1}$
\[
\begin{cases}
    \frac{dx_1}{dt} = m^2(-2x_1 + x_2) + x_1^2, \\
    \hspace{2.5mm}\vdots \\
    \frac{dx_k}{dt} = m^2(x_{k-1}-2x_k + x_{k+1}) + x_k^2, \\
    \hspace{2.5mm}\vdots \\
    \frac{dx_{m-1}}{dt} = m^2(x_{m-2}-2x_{m-1}) + x_{m-1}^2.
\end{cases}
\]
The dimension parameter $m \in \mathbb{N}$ is related to the spatial grid of the finite-difference approximation $(0, 1/m, \ldots, 1) \in \R^{m+1}$, and the boundary condition is \(x_0(t) = x_m(t) = 0\) for all $t\ge 0$.  
As we here need to discretize in both time and space we let \(x_k^n\) denote the numerical estimate of \(x_k(t_n)\) which again is an approximation of \(u(t_n, \frac{k}{m})\). 
We fix \(x^0 = \left[0, 100\sin\left(\frac{\pi}{m}\right), 100\sin\left(\frac{2\pi}{m}\right),\hdots, 100\sin\left(\frac{\pi(m-1)}{m}\right), 0\right]\). 

We want to numerically estimate the convergence rates for both refinements and for both adaptive and uniform time stepping. 
We therefore estimate \(\bar\tau\) for fixed \(m=32\) and let \(\epsilon\) vary between \(2^{-18}\) and \(2^{-25}\) in order to estimate the convergence rate when refining \(\epsilon\). 
In order to do the same for \(m\) we fix \(\epsilon=2^{-23}\) and let \(m \in \{4, 8, 16, 32, 64, 128, 216, 512\}\). 
Furthermore, experimentation shows that following the alternative choice for \(h\) given in Remark \ref{rem:altHRn} and setting \(h_n = \min\left\{\frac{\epsilon\sqrt{\abs{b(x)}}}{\sqrt{\abs{b'(x)b(x)}}}, \frac{1}{2m^2}\right\}\)  gives a significantly lower cost for \(m\geq16\).

The results of this study are given in Table \ref{tab:resultsFixedNVaryEps} and Table \ref{tab:resultsFixedEpsVaryN}. 
From the first table we see that when refining \(\epsilon\) the error of both methods appear to be of order \(\mathcal{O}(\epsilon)\). 
For the cost it looks like there is a log-term involved in the cost for the uniform version, while it is linear for the adaptive method. 
These results coincide with what we observed in earlier examples.
On the other hand, when refining \(n\) we note that the order of the cost and error looks to be quite similar for the two methods. The error has a rate of about \(\mathcal{O}(\epsilon^{1.5})\), while the cost is constant for the adaptive method, and actually slightly decreasing in the uniform case. 

\begin{table}
  \scriptsize
  \center
  \begin{tabular}{c | c c c | c c c}
    \hline
                 & \multicolumn{3}{c}{Adaptive} & \multicolumn{3}{c}{Uniform} \\
    \hline
    \(\epsilon\) & \(\bar\tau_\epsilon\) & \(\log_2(\abs{\bar\tau_{\epsilon}-\bar\tau_{2\epsilon}})\) & \(\log_{2}(N)\) & \(\bar\tau_\epsilon\) & \(\log_2(\abs{\bar\tau_{\epsilon}-\bar\tau_{2\epsilon}})\) & \(\log_{2}(N)\) \\
    \hline
    \(2^{-18}\) & 0.010977404445 &        & 16.18 & 0.010977979544 &        & 16.75 \\
    \(2^{-19}\) & 0.010977205587 & -22.26 & 17.18 & 0.010977471323 & -20.91 & 17.85 \\
    \(2^{-20}\) & 0.010977106560 & -23.27 & 18.18 & 0.010977244134 & -22.07 & 18.95 \\
    \(2^{-21}\) & 0.010977056824 & -24.26 & 19.18 & 0.010977125526 & -23.01 & 20.04 \\
    \(2^{-22}\) & 0.010977031941 & -25.26 & 20.18 & 0.010977063543 & -23.94 & 21.12 \\
    \(2^{-23}\) & 0.010977019507 & -26.26 & 21.18 & 0.010977035499 & -25.09 & 22.20 \\
    \(2^{-24}\) & 0.010977013282 & -26.26 & 22.18 & 0.010977021103 & -26.05 & 23.27 \\
    \(2^{-25}\) & 0.010977010170 & -27.26 & 23.18 & 0.010977014278 & -27.13 & 24.34 \\
    \hline    
  \end{tabular}
  \caption{Approximations of \(\bar\tau\) for fixed \(m = 32\) and varying \(\epsilon\)}
  \label{tab:resultsFixedNVaryEps}
\end{table}

\begin{table}
  \scriptsize
  \center
  \begin{tabular}{c | c c c | c c c}
    \hline
          & \multicolumn{3}{c}{Adaptive} & \multicolumn{3}{c}{Uniform} \\
    \hline
    \(m\) & \(\bar\tau_n\) & \(\log_2\left(\abs{\bar\tau_{n}-\bar\tau_{\frac{n}{2}}}\right)\) & \(\log_{2}(N)\) & \(\bar\tau_n\) & \(\log_2\left(\abs{\bar\tau_{n}-\bar\tau_{\frac{n}{2}}}\right)\) & \(\log_{2}(N)\) \\
    \hline
    \(4\)   & 0.010702612035 &        & 21.20 & 0.010702625433 &        & 22.27 \\
    \(8\)   & 0.010884845920 & -12.42 & 21.19 & 0.010884861297 & -12.42 & 22.26 \\
    \(16\)  & 0.010956076712 & -13.78 & 21.18 & 0.010956091321 & -13.78 & 22.23 \\
    \(32\)  & 0.010977019507 & -15.54 & 21.18 & 0.010977035499 & -15.54 & 22.20 \\
    \(64\)  & 0.010982686657 & -17.43 & 21.18 & 0.010982703459 & -17.43 & 22.16 \\
    \(128\) & 0.010984182464 & -19.35 & 21.18 & 0.010984200534 & -19.35 & 22.12 \\
    \(256\) & 0.010984572301 & -21.29 & 21.18 & 0.010984589827 & -21.29 & 22.08 \\
    \(512\) & 0.010984672958 & -23.24 & 21.18 & 0.010984691023 & -23.24 & 22.04 \\
    \hline    
  \end{tabular}
  \caption{Approximations of \(\bar\tau\) for fixed \(\epsilon = 2^{-23}\) and varying \(m\)}
  \label{tab:resultsFixedEpsVaryN}
\end{table}

\section{Conclusion and future work}
We have here presented two a priori adaptive numerical methods for estimating the blow-up time of ODEs. 
For both methods, we prove that under sufficient regularity, 
an $\mathcal{O}(\epsilon)$ approximation error 
is achieved at $\mathcal{O}(\epsilon^{-1})$ computational cost.
These theoretical findings are validated in several numerical experiments.
In addition, we show how the method can be applied to some problems not satisfying Assumption \ref{asmp:b.integrabilityCond}, and discuss how to extend adaptive time stepping from first-order 
to higher-order numerical integrators. 

There are several possible directions for future research. 
One such direction is to try to relax the conditions on \(b\) in both the 1-dimensional and multidimensional case.
For example, one might look into the possibility of remove the monotonicity assumption, allowing for more interesting dynamics, such as ODEs generating oscillating paths that explode in finite time. 
Additionally, adapting our numerical method to a stochastic setting trying to estimate the blow-up time of stochastic differential equation, such as in \cite{Dávila2005}, could also prove fruitful. A potential way forward is to combine our method with existing adaptive method for approximating hitting times of SDEs~\cite{Hoel2025}, relating to the threshold hitting time $\tau_\epsilon$ applied in the current work.  

\section*{Code availability}
The code used to run the numerical experiments provided in this work can be found at: \href{https://github.com/johannesvm/adaptive-methods-blow-up-times-ODEs}{https://github.com/johannesvm/adaptive-methods-blow-up-times-ODEs}

\bibliographystyle{amsplain}
\bibliography{References.bib}

\end{document}